\documentclass[12pt,reqno]{amsart}
\usepackage[a4paper,left=30mm,right=30mm,top=30mm,bottom=30mm,marginpar=25mm]{geometry}
\usepackage{amsmath, amssymb}
\usepackage{amsfonts}
\usepackage{amsthm}
\usepackage{amsthm}
\usepackage{mathrsfs}
\usepackage{comment}
\usepackage{color}
\usepackage[normalem]{ulem} 
\usepackage{enumerate}

\usepackage[nocompress, space]{cite}
\usepackage[displaymath,mathlines]{lineno}
\usepackage{eurosym}
\usepackage[dvips]{graphics}
\usepackage{graphicx}
\usepackage{epsfig}
\usepackage{mathtools}

\usepackage[dvipdfmx]{hyperref}
\usepackage[nocompress, space]{cite}
\usepackage{caption}
\usepackage{subcaption}

\newcommand{\ep}{\varepsilon}

\newcommand{\R}{\mathbb{R}}

\def\coloneqq{\mathrel{\mathop:}=}%
\def\revcoloneqq{=\mathrel{\mathop:}}%

\newcommand{\BUC}{{\rm BUC\,}}
\newcommand{\USC}{{\rm USC\,}}

\newcommand{\Lip}{{\rm Lip\,}}

\theoremstyle{definition}

\newtheorem{Th}{Theorem}[section]
\newtheorem{Prop}[Th]{Proposition}
\newtheorem{Lem}[Th]{Lemma}
\newtheorem{Def}{Definition}
\newtheorem{Rem}{Remark}

\numberwithin{equation}{section}

\DeclareMathOperator*{\argmax}{arg\,max} 

\title[Rate of convergence in homogenization]{On the rate of convergence in homogenization \\ of time-fractional Hamilton-Jacobi equations}

\author{Hiroyoshi Mitake, Shoichi Sato}
\address{Graduate School of Mathematical Sciences, University of Tokyo, 3-8-1 Komaba,
Meguroku, Tokyo, Japan 153-8914}
\email{mitake@ms.u-tokyo.ac.jp, shoichi@ms.u-tokyo.ac.jp}
\thanks{HM was partially supported by the JSPS grants: KAKENHI 
\#22K03382, 
\#21H04431, 
\#20H01816, 
\#19K03580, 
\#19H00639. 
SS was supported by the Leading Graduate Course for Frontiers of Mathematical Sciences and Physics, The University of Tokyo, MEXT
}

\date{\today}

\begin{document}
\maketitle

\begin{abstract}
Here, we consider periodic homogenization for time-fractional Hamilton--Jacobi equations. By using the perturbed test function method, we establish the convergence, and give estimates on a rate of convergence.
A main difficulty is the incompatibility between the function used in the doubling variable method, and the non-locality of the Caputo derivative.
Our approach is to provide a lemma to prove the rate of convergence without the doubling variable method with respect to the time variable, which is a key ingredient. 
\end{abstract}


\section{Introduction}

Let $T > 0$ and $\alpha \in (0, 1)$ be given constants and $\varepsilon > 0$ be a parameter.
We are concerned with time-fractional Hamilton-Jacobi equations:
\begin{equation} \label{TF_HJ_eqs}
\left \{
\begin{array}{ll}
\displaystyle \partial_{t}^{\alpha}u^{\varepsilon} + H \left( \frac{x}{\varepsilon}, Du^{\varepsilon} \right) \ =  0 & \mbox{in} \quad \mathbb{R}^N \times (0,T), \\
u^\varepsilon(x,0) = u_{0}(x) & \mbox{on} \quad \mathbb{R}^N.
\end{array}
\right.
\end{equation}
Here, $u^\varepsilon:\mathbb{R}^N \times [0,T] \to \mathbb{R}$ is an unknown function.
The Hamiltonian $H:\mathbb{R}^N \times \mathbb{R}^N \to \mathbb{R}$, the initial function $u_0:\mathbb{R}^N \to \mathbb{R}$ are given continuous functions, which satisfy
\begin{enumerate}
\item[(A1)]
 The function $H$ is uniformly coercive in the $y$-variable, i.e.,
\begin{equation*}
\displaystyle \lim_{r \rightarrow \infty} \inf \{ H(y, p) \mid y \in \mathbb{R}^N, |p| \geq r \} = \infty.
\end{equation*}
\item[(A2)]
 The function $y \mapsto H(y, p)$ is $\mathbb{Z}^N$-periodic, i.e.,
\begin{equation*}
H(y, p) = H(y+k, p) \quad \mbox{for} \, \, y, p \in \mathbb{R}^N, k \in \mathbb{Z}^N.
\end{equation*}
\item[(A3)] $u_0 \in \Lip(\mathbb{R}^N) \cap \BUC(\mathbb{R}^N)$, where we denote by $\BUC(\mathbb{R}^N)$ 
the set of all bounded, uniformly continuous functions on $\R^N$. 
\end{enumerate}
Moreover, we denote by $\partial_{t}^{\alpha}u^\varepsilon$ the \textit{Caputo fractional derivative} of $u^\varepsilon$ with respect to $t$, that is, 
\begin{equation*}
\displaystyle \partial_{t}^{\alpha}u^\varepsilon(x,t) \coloneqq 
\frac{1}{\Gamma(1 - \alpha)} \int_0^{t} (t-s)^{-\alpha} \partial_s u^\varepsilon(x,s) \,ds
\end{equation*}
for all $(x, t)\in \mathbb{R}^N \times (0,T)$, where $\Gamma$ is the Gamma function.
If we formally consider $\alpha=1$, then the Caputo fractional derivative coincides with the normal derivative with respect to $t$, and thus \eqref{TF_HJ_eqs} turns out to be the standard Hamilton--Jacobi equation. 

Fractional derivatives attracted great interest from both mathematics and applications within the last few decades, and developed in wide fields (see \cite{MK, NSY, KRY} for instance).  Studying differential equations with fractional derivatives is motivated by mathematical models that describe diffusion phenomena in complex media like fractals, which is sometimes called \textit{anomalous diffusion}. It has inspired further research on numerous related topics. 

The well-posedness of viscosity solutions to \eqref{TF_HJ_eqs} was established by \cite{GN, N, TY}. 
More precisely, a comparison principle, Perron's method, and stability results have been established in \cite{GN, N, TY}.  
We also mention that the equivalence of two weak solutions for linear uniformly parabolic equations  with Caputo's time-fractional derivatives was proved in \cite{GMS} by using the resolvent type approximation introduced by \cite{GLM}. 
In \cite{LTY}, the authors studied the regularity of viscosity solutions of \eqref{TF_HJ_eqs}, and got the large-time asymptotic result in some special settings. 


In our paper, we are interested in the asymptotic behavior of $u^\varepsilon$ as $\varepsilon \to 0$.
This singular limit problem is called ``\textit{homogenization problem}" with a background of the material science. 
Lions, Papanicolau and Varadhan in \cite{LPV} was the first who started to study homogenization for Hamilton--Jacobi type equations, and after the perturbed test function method was introduced by Evans \cite{Evans, Evans2}, there has been much literature concerning on homogenization for nonlinear partial differential equations. 
It was proved in \cite{LPV} that, when $\alpha=1$, under assumptions (A1)--(A3), $u^\ep$ converges to $\overline{u}$ locally uniformly on $\R^N \times [0,T]$ as $\ep \to 0$, 
and $\overline{u}$ solves the effective equation \eqref{TF_HJ} with $\alpha=1$. 
The effective Hamiltonian $\overline{H} \in C(\R^N)$ is determined in a nonlinear way by $H$ as following.  
For each $p \in \R^N$, it was shown in \cite{LPV} that there exists a unique constant $\overline{H}(p)\in \R$, which is called the \textit{effective Hamiltonian}, such that the stationary problem has a continuous viscosity solution to 
\begin{equation} \label{TF_CP}
H(y, p+Dv(y, p)) = \overline{H}(p) \ \mbox{in} \ \mathbb{T}^N. 
\end{equation}
If needed, we write $v=v(y,p)$ to clearly demonstrate the nonlinear dependence of $v$ on $p$. 
We call a problem to find a pair $(v(\cdot, p), \overline{H}(p)) \in \Lip(\mathbb{T}^N) \times \mathbb{R}$ to satisfy \eqref{TF_CP} the \textit{cell problem}. 
It is worth mentioning that in general $v(y,p)$ is not unique even up to additive constants.

Heuristically, owing to the two-scale  (i.e., $x$ and $\frac{x}{\varepsilon}$) asymptotic expansion, 
of solutions $u^\varepsilon$ to \eqref{TF_HJ_eqs} of the form
\begin{equation} \label{Two_scale_ex}
u^\varepsilon(x, t) = u^0\left(x, \frac{x}{\varepsilon}, t\right) + \varepsilon u^1\left(x, \frac{x}{\varepsilon}, t\right) + \varepsilon^2 u^2\left(x, \frac{x}{\varepsilon}, t\right) + \cdots, 
\end{equation} 
by plugging this into \eqref{TF_HJ_eqs}, and using the coercivity of Hamiltonian, we can naturally expect that 
\begin{equation*}
u^\varepsilon(x, t) = \overline{u}(x, t) + \varepsilon v\left(\frac{x}{\varepsilon}\right) + \mathcal{O}(\varepsilon^2).
\end{equation*}
Plugging this into \eqref{TF_HJ_eqs} again, 
letting $v$ be a solution of \eqref{TF_CP} with $p= D_x\overline{u}(x, t)$, 
we can expect that $u^\varepsilon$ converges to the limit function $\overline{u}$ as $\varepsilon \to 0$ which is a solution to
\begin{equation} \label{TF_HJ}
\left \{
\begin{array}{ll}
\displaystyle \partial_{t}^{\alpha}\overline{u} + \overline{H} ( D\overline{u} ) \ =  0 & \mbox{in} \, \, \mathbb{R}^N \times (0,T), \\
\overline{u}(x,0) = u_{0}(x) & \mbox{on} \, \, \mathbb{R}^N.
\end{array}
\right.
\end{equation}

%
%

Our main goal of this paper is to establish the convergence result of $u^\ep$ in $C(\R^N\times[0,T])$, and obtain a rate of convergence of $u^\ep$ to $u$, that is, an estimate for $\|u^\ep-u\|_{L^\infty(\R^N \times [0,T])}$ for any given $T>0$ with resect to $\ep$.  
We give two main results.
\begin{Th} \label{TF_conv_Main_Theorem_0}
Assume (A1) - (A3) hold.
Let $u^\varepsilon$ be the viscosity solution to \eqref{TF_HJ_eqs}.
The function $u^\varepsilon$ converges locally uniformly in $\mathbb{R}^N \times [0, T]$ to the function $\overline{u} \in C(\mathbb{R}^N \times [0, T])$ as $\varepsilon \to 0$, where $\overline{u}$ is the unique viscosity solution to \eqref{TF_HJ}.
\end{Th}
\begin{Th} \label{TF_conv_Main_Theorem}
Assume (A1)-(A3) hold.
We additionally assume
\begin{enumerate}
\item[(A4)] There exists $C > 0$ such that
\begin{equation*}
|H(x, p) - H(y, q)| \leq C(|x - y| + |p - q|) \quad \mbox{for} \, \, x, y, p, q \in \mathbb{R}^N.
\end{equation*}
\item[(A5)] There exists $C > 0$ such that
\begin{equation*}
H(x, p) \geq C^{-1}|p| - C \quad \mbox{for} \, \, x, p \in \mathbb{R}^N.
\end{equation*}
\end{enumerate}
Let $u^\varepsilon$ and $\overline{u}$ be the viscosity solutions of (\ref{TF_HJ_eqs}) and (\ref{TF_HJ}), respectively.
For each $\nu \in (0, 1)$, there exists a constant $C > 0$ such that
\begin{equation*}
\displaystyle \sup_{(x, t) \in \mathbb{R}^N \times (0, T)} |u^\varepsilon(x, t) - \overline{u}(x, t)| \leq C\varepsilon^{\frac{1}{3(2-\nu)}}.
\end{equation*}
\end{Th}

By the standard perturbed test function method introduced by \cite{Evans}, it is not hard to obtain Theorem \ref{TF_conv_Main_Theorem_0}.
To obtain Theorem \ref{TF_conv_Main_Theorem}, we use the method introduced by \cite{CDI} which is a combination of the perturbed test function method and the discount approximation.
A main difficulty we face here is the incompatibility between the doubling variable method which is often used in the theory of viscosity solutions and the non-locality of the Caputo derivative. 
More precisely, setting $\varphi(t,s)=(t-s)^2$, usually we have $\varphi_{t}(t,s)=-\varphi_{s}(t,s)$, which is an elementary, but important trick in the theory of viscosity solutions, however, we have 
\[
\partial_{t}^{\alpha}\varphi(t,s)\not=-\partial_{s}^{\alpha}\varphi(t,s)\quad\text{for all} \ t, s>0. 
\] 
Our approach is to provide a lemma to prove Theorem \ref{TF_conv_Main_Theorem} without using the doubling variable method with respect to time variable (see Lemma \ref{TF_Conv_Lem}), which is our key ingredient of this paper. 
It is inspired by \cite[Lemma 3.4]{GN}.

The study of a rate of convergence in homogenization of Hamilton--Jacobi equations 
was started by Capuzzo-Dolcetta and Ishii \cite{CDI}, 
and in recent years, there has been much interest on the optimal convergence rate. 
Mitake, Tran, and Yu \cite{MTY} was the first who established the optimal convergence rate
for convex Hamilton--Jacobi equations with several conditional settings 
by using the representation formula  for $u^\ep$ from optimal control theory, and weak KAM methods. 
Then, Cooperman \cite{Cooper1} obtained a near optimal convergence rate $|u^\ep(x,t)-\overline{u}(x,t)|\le C\log(C+e^{-1}t)$ when $n\ge3$ for general convex settings by using a theorem of Alexander \cite{A},  originally proved in the context of first-passage percolation. 
Finally, Tran and Yu \cite{TrY} has established the optimal rate $\mathcal{O}(\ep)$ by using an argument of Burago \cite{B}, which concludes the study of this whole program in the convex setting. 
It is worth mentioning that the best known result on a rate of convergence on $u^\ep$ for the stationary Hamilton--Jacobi equation with the general nonconvex Hamiltonian is still $\mathcal{O}(\varepsilon^\frac{1}{3})$ obtained by \cite{CDI}. 
The argument in \cite{CDI} can be easily adapted to \eqref{TF_HJ_eqs} with $\alpha=1$
(see \cite[Theorem 4.37]{T}). 
We also refer to \cite{Tu, HJ, NT} for other development on this subject. 
In all works \cite{MTY, Tu, Cooper1, TrY, HJ, NT}, the argument relies on the optimal control formula 
for $u^\ep$, and therefore it is seemingly rather challenging to obtain the optimal rate of convergence 
of \eqref{TF_HJ_eqs} for $0 < \alpha < 1$, which remains completely open.

\medskip
\noindent
\textbf{Organization of the paper}. 
The paper is organized as follows. 
In Section 2, we recall the definition of viscosity solutions to \eqref{TF_HJ_eqs}, 
and give several basic results on time-fractional Hamilton--Jacobi equations in the theory of viscosity solutions. 
In Section 3, we give regularity results of viscosity solutions to \eqref{TF_HJ_eqs}, and \eqref{TF_HJ}. 
Sections 4 and 5 are devoted to prove the convergence of $u^\ep$, Theorem \ref{TF_conv_Main_Theorem_0}, 
and a convergence rate of $u^\varepsilon$ to $\overline{u}$ in $L^\infty$, Theorem \ref{TF_conv_Main_Theorem}.

\section{Preliminaries}

In this section, we recall the definition of viscosity solutions. 
First, we give several elementary facts of the Caputo derivative. 
These are well-known, but we give proofs to make the paper self-contained. 

\begin{Prop} \label{Weak_Caputo}
Let $f:[0, T] \mapsto \mathbb{R}$ be a function such that $f \in C^1((0, T]) \cap C([0, T])$ and $f^\prime \in L^1(0, T)$.
Then,
\begin{align}
\partial_t^\alpha f(t) &= \frac{f(t) - f(0)}{t^\alpha \Gamma(1 - \alpha)} + \frac{\alpha}{\Gamma(1 - \alpha)} \int_0^t \frac{f(t) - f(t-\tau)}{\tau^{\alpha + 1}} d\tau \label{Caputo_1}\\
&= \frac{\alpha}{\Gamma(1 - \alpha)} \int_{-\infty}^t \frac{f(t) - \tilde{f}(\xi)}{|t - \xi|^{\alpha + 1}} d\xi \notag
\end{align}
for any $t \in (0, T]$, where $\tilde{f}$ is defined by
\begin{equation*}
\tilde{f}(t) \coloneqq 
\left\{ 
\begin{array}{ll}
f(t) & \mbox{in} \, \, (0, T], \\
f(0) & \mbox{in} \, \, (-\infty, 0].
\end{array} \right.
\end{equation*}
\end{Prop}

\begin{proof}
For $\varepsilon > 0$, using the integration by parts, we have
\begin{align*}
\int_0^{t-\varepsilon} (t-s)^{-\alpha} \partial_s f(s) ds &= \int_\varepsilon^t \tau^{-\alpha} \partial_{\tau}(-f(t - \tau)) d\tau \\
&= [\tau^{-\alpha}(-f(t-\tau))]_\varepsilon^t - \int_\varepsilon^t -\alpha \tau^{-(1+\alpha)} (-f(t - \tau)) d\tau \\
&= -\frac{f(0)}{t^\alpha} + \frac{f(t - \varepsilon)}{\varepsilon^\alpha} + \alpha \int_\varepsilon^t \frac{-f(t - \tau)}{\tau^{1+\alpha}} d\tau.
\end{align*}
Note that
\begin{equation*}
\alpha \int_\varepsilon^t \frac{f(t)}{\tau^{1+\alpha}} d\tau = - \frac{f(t)}{t^\alpha} + \frac{f(t)}{\varepsilon^\alpha}.
\end{equation*}
Thus, 
\begin{align*}
\int_0^{t-\varepsilon} (t-s)^{-\alpha} \partial_s f(s) ds &= \frac{f(t) - f(0)}{t^\alpha} + \alpha \int_\varepsilon^t \frac{f(t) - f(t - \tau)}{\tau^{1+\alpha}} d\tau - \frac{f(t) - f(t-\varepsilon)}{\varepsilon^\alpha}.
\end{align*}
Using the smoothness of $f$ and taking the limit as $\varepsilon \to 0$ above, we get \eqref{Caputo_1}.

On the other hand, using the change of variables, we obtain
\begin{align*}
\int_{-\infty}^t \frac{f(t) - \tilde{f}(\xi)}{|t - \xi|^{\alpha + 1}} d\xi 
&= \int_{-\infty}^0 \frac{f(t) - f(0)}{|t - \xi|^{\alpha + 1}} d\xi + \int_0^t \frac{f(t) - f(\xi)}{|t - \xi|^{\alpha + 1}} d\xi \\
&= \frac{f(t) - f(0)}{\alpha t^\alpha} + \int_0^t \frac{f(t) - f(t-\tau)}{\tau^{\alpha + 1}} d\tau,
\end{align*}
which completes the proof.
\end{proof}

From the observation in Proposition \ref{Weak_Caputo}, we set
\begin{align*}
&\displaystyle J[f](t) \coloneqq \frac{f(t) - f(0)}{t^\alpha \Gamma(1 - \alpha)}, \\
&\displaystyle K_{(a, b)}[f](t) \coloneqq \frac{\alpha}{\Gamma(1 - \alpha)} \int_a^b \frac{f(t) - f(t-\tau)}{\tau^{\alpha + 1}} d\tau
\end{align*}
for $t \in (0, T]$ and $0 \leq a < b \leq t$.

\begin{Prop}\label{K_conti}
Let $f \in C^1((0, T]) \cap C([0, T])$.
For any $t \in (0, T]$, the function $K_{(0, t)}[f](t)$ exists and continuous in $(0, T]$.
\end{Prop}

\begin{proof}
For any $s \in (0, T]$, we set 
\begin{equation*}
g_s(\tau) \coloneqq \frac{f(s) - f(s-\tau)}{\tau^{\alpha + 1}} \chi_{(0, s)}(\tau) \quad \mbox{for} \, \, \tau \in (0, s),
\end{equation*}
where $\chi_I$ is the characteristic function, i.e., $\chi_{I}(\tau) = 1$ if $\tau \in I$ and $\chi_{I}(\tau) = 0$ if $\tau \notin I$ for an interval $I$.
Fix $t \in (0, T]$.
It suffices to prove that $g_s \to g_t$ in $L^1(0, T)$ as $s \to t$. 
Let $\rho\in(0,\frac{t}{2})$, and for $s \in (t - \rho, \min \{t + \rho, T \})$, we have
\begin{align*}
|g_s(\tau)| &\leq \frac{|f(s) - f(s-\tau)|}{\tau^{\alpha + 1}} \chi_{(0, \rho)}(\tau) + \frac{|f(s) - f(s-\tau)|}{\tau^{\alpha + 1}} \chi_{(\rho, s)}(\tau) \\
&\leq \frac{\sup_{\tau \in (t-2\rho, \min \{t + \rho, T \})} |f^\prime(\tau)|}{\tau^{\alpha}} \chi_{(0, \rho)}(\tau) + \frac{2\|f\|_\infty}{\tau^{\alpha + 1}} \chi_{(\rho, T)}(\tau) \quad \mbox{for} \, \, \tau \in (0, T).
\end{align*}
The right hand side is integrable on $(0, T)$.
Thus, using the dominated convergence theorem, we get the desired result.
\end{proof}

Now, we define the definition of viscosity solutions to \eqref{TF_HJ_eqs}.

\begin{Def} \label{Def_vis_sol}
An upper semicontinuous function $u^\varepsilon : \mathbb{R}^N \times [0, T] \mapsto \mathbb{R}$ is said to be a viscosity subsolution of (\ref{TF_HJ_eqs}) if for any $\varphi \in C^1(\mathbb{R}^N \times (0, T]) \cap C(\mathbb{R}^N \times [0, T])$ one has
\begin{equation*}
\displaystyle J[\varphi](\hat{x}, \hat{t}) + K_{(0, \hat{t})}[\varphi](\hat{x}, \hat{t}) + H \left( \frac{\hat{x}}{\varepsilon}, D\varphi(\hat{x}, \hat{t}) \right) \leq 0,
\end{equation*}
whenever $u^\varepsilon - \varphi$ attains a maximum at $(\hat{x}, \hat{t}) \in \mathbb{R}^N \times (0, T]$ over $\mathbb{R}^N \times [0, T]$ and $u^\varepsilon(\cdot, 0) \leq u_0$ on $\mathbb{R}^N$.

Similarly, a lower semicontinuous function $u^\varepsilon : \mathbb{R}^N \times [0, T] \mapsto \mathbb{R}$ is said to be a viscosity supersolution of (\ref{TF_HJ_eqs}) if for any $\varphi \in C^1(\mathbb{R}^N \times (0, T]) \cap C(\mathbb{R}^N \times [0, T])$ one has
\begin{equation*}
\displaystyle J[\varphi](\hat{x}, \hat{t}) + K_{(0, \hat{t})}[\varphi](\hat{x}, \hat{t}) + H \left( \frac{\hat{x}}{\varepsilon}, D\varphi(\hat{x}, \hat{t}) \right) \geq 0,
\end{equation*}
whenever $u^\varepsilon - \varphi$ attains a minimum at $(\hat{x}, \hat{t}) \in \mathbb{R}^N \times (0, T]$ over $\mathbb{R}^N \times [0, T]$ and $u^\varepsilon(\cdot, 0) \geq u_0$ on $\mathbb{R}^N$.

Finally, we call $u^\varepsilon \in C(\mathbb{R}^N \times [0, T])$ a viscosity solution of (\ref{TF_HJ_eqs}) if $u^\varepsilon$ is both a viscosity subsolution and supersolution of (\ref{TF_HJ_eqs}).
\end{Def}

We give several equivalent conditions for viscosity subsolutions/supersolutions to \eqref{TF_HJ_eqs}, which are  often used in Sections 3--5.

\begin{Prop} \label{vissol_equivalent}
For any $\varepsilon > 0$, let $u^\varepsilon : \mathbb{R}^N \times [0, T] \mapsto \mathbb{R}$ be an upper semicontinuous function.
Then, the following statements are equivalent.
\begin{enumerate}
\item [(a)] $u^\varepsilon$ is a viscosity subsolution of (\ref{TF_HJ_eqs}).
\item [(b)] For any $\varphi \in C^1(\mathbb{R}^N \times (0, T]) \cap C(\mathbb{R}^N \times [0, T])$ one has
\begin{equation*}
\displaystyle J[u^\varepsilon](\hat{x}, \hat{t}) + K_{(0, \rho)}[\varphi](\hat{x}, \hat{t}) + K_{(\rho, \hat{t})}[u^\varepsilon](\hat{x}, \hat{t}) + H \left( \frac{\hat{x}}{\varepsilon}, D\varphi(\hat{x}, \hat{t}) \right) \leq 0 \end{equation*}
for $0 < \rho < \hat{t}$, 
whenever $u^\varepsilon - \varphi$ attains a maximum at $(\hat{x}, \hat{t}) \in \mathbb{R}^N \times (0, T]$ over $\mathbb{R}^N \times [0, T]$ and $u^\varepsilon(\cdot, 0) \leq u_0$ on $\mathbb{R}^N$.
\item [(c)] $K_{(0, \hat{t})}[u^\varepsilon](\hat{x}, \hat{t})$ exists and for any $\varphi \in C^1(\mathbb{R}^N \times (0, T]) \cap C(\mathbb{R}^N \times [0, T])$ one has
\begin{equation*}
\displaystyle J[u^\varepsilon](\hat{x}, \hat{t}) + K_{(0, \hat{t})}[u^\varepsilon](\hat{x}, \hat{t}) + H \left( \frac{\hat{x}}{\varepsilon}, D\varphi(\hat{x}, \hat{t}) \right) \leq 0,
\end{equation*}
whenever $u^\varepsilon - \varphi$ attains a local maximum at $(\hat{x}, \hat{t}) \in \mathbb{R}^N \times (0, T]$ over $\mathbb{R}^N \times [0, T]$ and $u^\varepsilon(\cdot, 0) \leq u_0$ on $\mathbb{R}^N$.
\end{enumerate}
Similarly, we obtain equivalent conditions for viscosity supersolutions. 
\end{Prop}

\begin{proof}
(a) $\Rightarrow$ (b). 
Take $\varphi \in C^1(\mathbb{R}^N \times (0, T]) \cap C(\mathbb{R}^N \times [0, T])$ such that $u^\varepsilon - \varphi$ has a strict maximum at $(\hat{x}, \hat{t}) \in \mathbb{R}^N \times (0, T]$.
Without loss of generality, we may assume that $(u^\varepsilon - \varphi)(\hat{x}, \hat{t}) = 0$.
Fix $0 < \rho < \hat{t}$.
For any $\sigma > 0$, take a function $\varphi_\sigma \in C^1(\mathbb{R}^N \times (0, T]) \cap C(\mathbb{R}^N \times [0, T])$ satisfying
\begin{itemize}
\item $\varphi_\sigma = \varphi$ in $B((\hat{x}, \hat{t}), \frac{\rho}{2})$,
\item$u^\varepsilon \leq \varphi_\sigma \leq \varphi$ in $B((\hat{x}, \hat{t}), \rho)$,
\item$u^\varepsilon \leq \varphi_\sigma \leq u^\varepsilon + \sigma$ in $\big(\R^N\times[0,T]\big)\setminus B((\hat{x}, \hat{t}), \rho)$.
\end{itemize}
Noting that $\max(u^\varepsilon - \varphi_\sigma) = (u^\varepsilon - \varphi_\sigma)(\hat{x}, \hat{t})$, by (a), we have
\begin{equation} \label{vis_def_equi_a}
\displaystyle J[\varphi_\sigma](\hat{x}, \hat{t}) + K_{(0, \hat{t})}[\varphi_\sigma](\hat{x}, \hat{t}) + H \left( \frac{\hat{x}}{\varepsilon}, D\varphi_\sigma(\hat{x}, \hat{t}) \right) \leq 0.
\end{equation}
It is clear that $\lim_{\sigma \to 0} J[\varphi_\sigma](\hat{x}, \hat{t}) = J[u^\varepsilon](\hat{x}, \hat{t})$ and $D\varphi_\sigma(\hat{x}, \hat{t}) = D\varphi(\hat{x}, \hat{t})$.
Noting that $\varphi(\hat{x}, \hat{t}) - \varphi_\sigma(\hat{x}, \hat{t}) \leq \varphi(\hat{x}, \hat{t} - \tau) - \varphi_\sigma(\hat{x}, \hat{t} - \tau)$ for $0 < \tau < \rho$, we have
\begin{align*}
K_{(0, \hat{t})}[\varphi_\sigma](\hat{x}, \hat{t})
&=
K_{(0, \rho)}[\varphi_\sigma](\hat{x}, \hat{t}) + K_{(\rho, \hat{t})}[\varphi_\sigma](\hat{x}, \hat{t}) \\
&\geq
K_{(0, \rho)}[\varphi](\hat{x}, \hat{t}) + K_{(\rho, \hat{t})}[\varphi_\sigma](\hat{x}, \hat{t}).
\end{align*}
By the dominated convergence theorem, we have 
\begin{equation*}
\lim_{\sigma \to 0} K_{(\rho, \hat{t})}[\varphi_\sigma](\hat{x}, \hat{t})
=
K_{(\rho, \hat{t})}[u^\varepsilon](\hat{x}, \hat{t}).
\end{equation*}
Thus, sending $\sigma \to 0$ in \eqref{vis_def_equi_a}, we obtain
\begin{equation*}
\displaystyle J[u^\varepsilon](\hat{x}, \hat{t}) + K_{(0, \rho)}[\varphi](\hat{x}, \hat{t}) + K_{(\rho, \hat{t})}[u^\varepsilon](\hat{x}, \hat{t}) + H \left( \frac{\hat{x}}{\varepsilon}, D\varphi(\hat{x}, \hat{t}) \right) \leq 0.
\end{equation*}

(b) $\Rightarrow$ (c). 
Assume that there exists $\varphi \in C^1(\mathbb{R}^N \times (0, T]) \cap C(\mathbb{R}^N \times [0, T])$ such that $u^\varepsilon - \varphi$ has a strict maximum at $(\hat{x}, \hat{t}) \in \mathbb{R}^N \times (0, T)$.
By (b), we have
\begin{equation*}
\displaystyle J[u^\varepsilon](\hat{x}, \hat{t}) + K_{(0, \rho)}[\varphi](\hat{x}, \hat{t}) + K_{(\rho, \hat{t})}[u^\varepsilon](\hat{x}, \hat{t}) + H \left( \frac{\hat{x}}{\varepsilon}, D\varphi(\hat{x}, \hat{t}) \right) \leq 0,
\end{equation*}
which implies that
\begin{equation*}
\int_0^\rho \frac{\varphi(\hat{x}, \hat{t}) - \varphi(\hat{x}, \hat{t} - \tau)}{\tau^{\alpha + 1}} d\tau
+
\int_\rho^{\hat{t}} \frac{u^\varepsilon(\hat{x}, \hat{t}) - u^\varepsilon(\hat{x}, \hat{t} - \tau)}{\tau^{\alpha + 1}} d\tau
\leq C
\end{equation*}
for some $C > 0$ which is independent of $\rho > 0$.
Noting that $\varphi \in C^1(\mathbb{R}^N \times (0, T])$, we have
\begin{equation*}
\left| \int_0^\rho \frac{\varphi(\hat{x}, \hat{t}) - \varphi(\hat{x}, \hat{t} - \tau)}{\tau^{\alpha + 1}} d\tau \right|
\leq
\int_0^\rho \frac{| \varphi_t(\hat{x}, \hat{t})|}{\tau^{\alpha}} d\tau
\to 0 \quad \mbox{as} \,\, \rho \to 0.
\end{equation*}

For $r \in \mathbb{R}$, we write $r_+ \coloneqq \max \{r, 0 \}$, $r_- \coloneqq \min \{r, 0 \}$.
Then,
\begin{align}
&\int_\rho^{\hat{t}} \frac{u^\varepsilon(\hat{x}, \hat{t}) - u^\varepsilon(\hat{x}, \hat{t} - \tau)}{\tau^{\alpha + 1}} d\tau 
\nonumber\\
=&\, 
\int_\rho^{\hat{t}} \frac{(u^\varepsilon(\hat{x}, \hat{t}) - u^\varepsilon(\hat{x}, \hat{t} - \tau))_+}{\tau^{\alpha + 1}} - \frac{(u^\varepsilon(\hat{x}, \hat{t}) - u^\varepsilon(\hat{x}, \hat{t} - \tau))_-}{\tau^{\alpha + 1}} d\tau. 
\label{TF_K_exists}
\end{align}

Noting that $u^\varepsilon(\hat{x}, \hat{t}) - \varphi(\hat{x}, \hat{t}) \geq u^\varepsilon(\hat{x}, \hat{t} - \tau) - \varphi(\hat{x}, \hat{t} - \tau)$, we obtain
\begin{align}
C
\geq
- \int_\rho^{\hat{t}} \frac{(u^\varepsilon(\hat{x}, \hat{t}) - u^\varepsilon(\hat{x}, \hat{t} - \tau))_-}{\tau^{\alpha + 1}} d\tau 
&\geq
- \int_\rho^{\hat{t}} \frac{(\varphi(\hat{x}, \hat{t}) - \varphi(\hat{x}, \hat{t} - \tau))_-}{\tau^{\alpha + 1}} d\tau \nonumber
\\
&\geq - C^\prime 
\label{eq:bdd}
\end{align}
for some $C^\prime > 0$ which is independent of $\rho > 0$.
By the monotone convergence theorem, sending $\rho \to 0$ in \eqref{TF_K_exists}, we obtain
\begin{equation*}
\displaystyle J[u^\varepsilon](\hat{x}, \hat{t}) + K_{(0, \hat{t})}[u^\varepsilon](\hat{x}, \hat{t}) + H \left( \frac{\hat{x}}{\varepsilon}, D\varphi(\hat{x}, \hat{t}) \right) \leq 0.
\end{equation*}

(c) $\Rightarrow$ (a). 
Take $\varphi \in C^1(\mathbb{R}^N \times (0, T]) \cap C(\mathbb{R}^N \times [0, T])$ such that $u^\varepsilon - \varphi$ has a strict maximum at $(\hat{x}, \hat{t}) \in \mathbb{R}^N \times (0, T]$.
Noting that $(u^\varepsilon-\varphi)(\hat{x}, \hat{t}-\tau) \leq (u^\varepsilon-\varphi)(\hat{x}, \hat{t})$ for all $\tau \in [0, \hat{t}]$, we have
\begin{equation*}
J[\varphi](\hat{x}, \hat{t}) \leq J[u^\varepsilon](\hat{x}, \hat{t}),
K_{(0, \hat{t})}[\varphi](\hat{x}, \hat{t}) \leq K_{(0, \hat{t})}[u^\varepsilon](\hat{x}, \hat{t}).
\end{equation*}
This inequality together with (c), we get (a).
\end{proof}

We state basic results for time-fractional Hamilton-Jacobi equations, and we refer to \cite{N} for the proofs. 
\begin{Prop}[Comparison principle, {\rm \cite[Theorem 3.1]{GN}} ] \label{Comparison_principle} 
Let $u, v:\mathbb{R}^N \times [0, T] \to \mathbb{R}$ be a viscosity subsolution and a viscosity supersolution of (\ref{TF_HJ_eqs}), respectively.
If $u, v \in \BUC(\mathbb{R}^N \times [0, T])$ and $u(\cdot, 0) \leq v(\cdot, 0)$ on $\mathbb{R}^N$, then $u \leq v$ on $\mathbb{R}^N \times [0, T]$.
\end{Prop}

\begin{Prop}[Existence of a solution, {\rm\cite[Theorem 4.2]{GN}}] \label{Existence} 
Let $u_-, u_+:\mathbb{R}^N \times [0, T] \to \mathbb{R}$ be a viscosity subsolution and a viscosity supersolution of (\ref{TF_HJ_eqs}), respectively.
Suppose that $u_- \leq u_+$ in $\mathbb{R}^N \times [0, T] \to \mathbb{R}$.
There exists a viscosity solution $u$ of (\ref{TF_HJ_eqs}) that satisfies $u_- \leq u \leq u_+$ in $\mathbb{R}^N \times [0, T] \to \mathbb{R}$.
\end{Prop}

Proposition \ref{Existence} follows from Lemmas \ref{Closedness_supinf} and \ref{Perrons_method}.
We denote $S^-$ and $S^+$ by the set of all subsolutions and supersolutions of (\ref{TF_HJ_eqs}), respectively.

\begin{Lem}[Closedness under supremum/infimum operator, {\rm\cite[Lemma 4.1]{GN}}] \label{Closedness_supinf}
Let $X$ be a nonempty subset of $S^-$ (resp., $S^+$).
Set 
\begin{equation*}
\displaystyle u(x, t) \coloneqq \sup_{v \in X} v(x, t) \quad (resp., \inf_{v \in X} v(x, t)) \quad \mbox{for} \, \, (x, t) \in \mathbb{R}^N \times [0, T].
\end{equation*}
If $u < \infty$ (resp., $u > -\infty$) on $\mathbb{R}^N \times [0, T]$, then $u$ is a viscosity subsolution (resp., supersolution) of (\ref{TF_HJ_eqs}).
\end{Lem}

\begin{Lem} \label{Perrons_method} 
Let $u_+:\mathbb{R}^N \times [0, T] \to \mathbb{R}$ be a supersolution of (\ref{TF_HJ_eqs}).
Define
\begin{equation*}
X \coloneqq \{ v \in S^- \mid v \leq u_+ \, \mbox{in} \, \mathbb{R}^N \times [0, T] \}.
\end{equation*}
If $u \in X$ is not a supersolution of (\ref{TF_HJ_eqs}), then there exists a function $w \in X$ and a point $(y, s) \in \mathbb{R}^N \times (0, T]$ such that $u(y, s) < w(y, s)$.
\end{Lem}

\begin{proof}
Since $u$ is not a supersolution of (\ref{TF_HJ_eqs}), there exists $((\hat{x}, \hat{t}), \varphi) \in (\mathbb{R}^N \times (0, T]) \times (C^1(\mathbb{R}^N \times (0, T]) \cap C(\mathbb{R}^N \times [0, T]))$ such that
\begin{equation*}
(u - \varphi)(x, t) > (u - \varphi)(\hat{x}, \hat{t}) = 0 \quad \mbox{for all} \,\, (x, t) \neq (\hat{x}, \hat{t}),
\end{equation*}
and
\begin{equation*}
J[\varphi](\hat{x}, \hat{t}) + K_{(0, \hat{t})}[\varphi](\hat{x}, \hat{t}) + H \left( \frac{\hat{x}}{\varepsilon}, D\varphi(\hat{x}, \hat{t}) \right) < 0.
\end{equation*}
For $r > 0$ small enough, we have
\begin{equation} \label{Perron_contradiction}
J[\varphi](x, t) + K_{(0, t)}[\varphi](x, t) + H \left( \frac{x}{\varepsilon}, D\varphi(x, t) \right) \leq 0 \quad \mbox{for} \,\, B((\hat{x}, \hat{t}), 2r) \cap (\mathbb{R}^N \times (0, T]),
\end{equation}
where we denote by $B((\hat{x}, \hat{t}), 2r)$ a cylindrical neighborhood of $(\hat{x}, \hat{t})$.

It is clear to see that $\varphi \leq u \leq u_+$ in $\mathbb{R}^N \times [0, T]$. 
We set $\lambda \coloneqq \frac{1}{2}(u_+ - \varphi)(\hat{x}, \hat{t}) > 0$. 
Since $u_+ - \varphi$ is lower semicontinuous, if $r > 0$ is small enough, we have $\varphi + \lambda \leq u_+$ in $B((\hat{x}, \hat{t}), 2r)$.
Since $u > \varphi$ in $\mathbb{R}^N \times [0, T] \setminus \{(\hat{x}, \hat{t})\}$, there exists $\lambda^\prime \in (0, \lambda)$ such that $\varphi + 2\lambda^\prime \leq u$ in $\big(B(\hat{x}, 2r) \times [0, T]\big) \setminus B((\hat{x}, \hat{t}), r)$.
Define the function $w:\mathbb{R}^N \times [0, T] \mapsto \mathbb{R}$ by
\begin{equation*}
w \coloneqq 
\left\{ 
\begin{array}{ll}
\max \{u, \varphi + \lambda^\prime \} & \mbox{in} \, \, B((\hat{x}, \hat{t}), r), \\
u & \mbox{in} \, \, \big(\mathbb{R}^N \times [0, T]\big)\setminus B((\hat{x}, \hat{t}), r).
\end{array} \right.
\end{equation*}
We claim that $w \in X$.
It is easy to see that $w \leq u_+$ in $\mathbb{R}^N \times [0, T]$.
Take $\psi \in C^1(\mathbb{R}^N \times (0, T]) \cap C(\mathbb{R}^N \times [0, T])$ such that $w - \psi$ has a strict maximum at $(\hat{y}, \hat{s}) \in \mathbb{R}^N \times (0, T]$ with $(w-\psi)(\hat{y}, \hat{s})=0$.
If $w = u$ at $(\hat{y}, \hat{s})$, since $u$ is a subsolution, we are done.
We consider the case where $w = \varphi + \lambda$ at $(\hat{y}, \hat{s})$.
By definition of $w$, we have $(\hat{y}, \hat{s}) \in B((\hat{x}, \hat{t}), r)$ and $w(\hat{y}, \hat{s} - \tau) \geq \varphi(\hat{y}, \hat{s} - \tau) + \lambda^\prime$ for $\tau \in [0, \hat{s}]$.
Noting that $0 = (w - \psi)(\hat{y}, \hat{s}) \geq (w - \psi)(\hat{y}, \hat{s} - \tau)$, we have
\begin{align*}
&\psi(\hat{y}, \hat{s}) - \psi(\hat{y}, \hat{s} - \tau) \leq w(\hat{y}, \hat{s}) - w(\hat{y}, \hat{s} - \tau) \\
&\leq \varphi(\hat{y}, \hat{s}) + \lambda^\prime - (\varphi(\hat{y}, \hat{s} - \tau) + \lambda^\prime) 
= \varphi(\hat{y}, \hat{s}) - \varphi(\hat{y}, \hat{s} - \tau)
\end{align*}
for $\tau \in [0, \hat{s}]$ 
which yields $J[\psi](\hat{y}, \hat{s}) \leq J[\varphi](\hat{y}, \hat{s})$ and $K_{(0, \hat{s})}[\psi](\hat{y}, \hat{s}) \leq K_{(0, \hat{s})}[\varphi](\hat{y}, \hat{s})$.
Also, since $\varphi = \psi$ at $(\hat{y}, \hat{s})$ and $0 > w - \psi \geq \varphi + \lambda^\prime - \psi$ in $\big(\mathbb{R}^N \times [0, T]\big) \setminus \{ (\hat{y}, \hat{s}) \}$, we have $D\varphi = D\psi$ at $(\hat{y}, \hat{s})$.
Combining this with \eqref{Perron_contradiction}, we obtain
\begin{align*}
&J[\psi](\hat{y}, \hat{s}) + K_{(0, \hat{s})}[\psi](\hat{y}, \hat{s}) + H \left( \frac{\hat{y}}{\varepsilon}, D\psi(\hat{y}, \hat{s}) \right) \\
&\leq J[\varphi](\hat{y}, \hat{s}) + K_{(0, \hat{s})}[\varphi](\hat{y}, \hat{s}) + H \left( \frac{\hat{y}}{\varepsilon}, D\varphi(\hat{y}, \hat{s}) \right) \le0, 
\end{align*}
which implies that $w \in S^-$ and finishes the proof. 
\end{proof}

\section{Regularity estimates}

\begin{Prop} \label{TF_uni_bdd}
Assume (A1)--(A3) hold.
Let $u^\varepsilon$ be the viscosity solution of (\ref{TF_HJ_eqs}).
There exists a constant $M^\prime > 0$ depending only on $u_0, H$ and $T$ such that
\begin{equation}
|u^\varepsilon(x, t)| \leq M^\prime \quad \mbox{for} \quad (x, t) \in \mathbb{R}^N \times [0, T].
\end{equation}
\end{Prop}

\begin{proof}
We set $u^+(x, t) \coloneqq u_0(x) + M t^\alpha$ and $u^-(x, t) \coloneqq u_0(x) - M t^\alpha$ with
\begin{equation}\label{const-M}
M \coloneqq \frac{1}{\Gamma(1 - \alpha)} \max \{ H(\xi, p) \mid \xi \in \mathbb{R}^N, |p| \leq \Lip[u_0] \}.
\end{equation}
Then, $u^+$ and $u^-$ are a viscosity supersolution and a viscosity subsolution of (\ref{TF_HJ_eqs}), respectively.
By Proposition \ref{Comparison_principle}, we obtain
\begin{equation*}
u_0(x) - M t^\alpha \leq u^\varepsilon(x, t) \leq u_0(x) + M t^\alpha
\end{equation*}
on $\mathbb{R}^N \times [0, T]$.
Thus, we choose $M^{\prime} \coloneqq \|u_0\|_{\infty} + M T^\alpha$.
This completes the proof.
\end{proof}

\begin{Prop} \label{TF_uni_time_regularity}
Assume (A1)--(A3) hold.
Let $u^\varepsilon$ be the viscosity solution of (\ref{TF_HJ_eqs}).
There exists a constant $M > 0$ depending only on $u_0$ and $H$ such that
\begin{equation} \label{TF_alpha_holder}
|u^\varepsilon(x, t) - u^\varepsilon(x, s)| \leq M |t - s|^\alpha
\end{equation}
for $(x, t, s) \in \mathbb{R}^N \times [0, T] \times [0, T]$.
\end{Prop}

\begin{proof}
Let $u^+$ be the function defined in the proof of Proposition \ref{TF_uni_bdd}, and let
\begin{equation*}
X \coloneqq \{ v \in S^- \mid v \leq u^+ \, \mbox{in} \, \mathbb{R}^N \times [0, T], \, v \, \mbox{satisfies} \, (\ref{TF_alpha_holder}) \},
\end{equation*}
where we denote $S^-$ by the set of all subsolutions of (\ref{TF_HJ_eqs}) and $M$ is the constant defined by \eqref{const-M}. 
Set $u(x, t) \coloneqq \sup_{v \in X} v(x, t)$ for $(x, t) \in \mathbb{R}^N \times [0, T]$. 
It is clear to see that $u$ satisfies \eqref{TF_alpha_holder}. 

Next, let us prove that $u$ is a viscosity solution of \eqref{TF_HJ_eqs}.
Noting that $u_0(x) - Mt^\alpha \in X$, by Lemma \ref{Closedness_supinf}, we see that $u$ is a viscosity subsolution of (\ref{TF_HJ_eqs}).
Suppose that $u$ is not a supersolution of (\ref{TF_HJ_eqs}).
Then there exists $((\hat{x}, \hat{t}), \varphi) \in (\mathbb{R}^N \times (0, T]) \times (C^1(\mathbb{R}^N \times (0, T]) \cap C(\mathbb{R}^N \times [0, T]))$ such that
\begin{equation*}
(u - \varphi)(x, t) > (u - \varphi)(\hat{x}, \hat{t}) = 0 \quad \mbox{for} \,\, (x, t) \neq (\hat{x}, \hat{t}),
\end{equation*}
and
\begin{equation*}
J[\varphi](\hat{x}, \hat{t}) + K_{(0, \hat{t})}[\varphi](\hat{x}, \hat{t}) + H \left( \frac{\hat{x}}{\varepsilon}, D\varphi(\hat{x}, \hat{t}) \right) < 0.
\end{equation*}
By Lemma \ref{Perrons_method}, there exists a viscosity subsolution $w \in \USC(\mathbb{R}^N \times [0, T])$ of \eqref{TF_HJ_eqs} satisfying $w \leq u^+$ on $\mathbb{R}^N \times [0, T]$ and $w(\hat{x}, \hat{t}) > u(\hat{x}, \hat{t})$. 
By construction of $w$ in the  proof of Lemma \ref{Perrons_method}, it is clear to see that $w$ satisfies \eqref{TF_alpha_holder}, and also $w \in X$. 
This immediately gives a contradiction. 
By the uniqueness of viscosity solutions of \eqref{TF_HJ_eqs}, we obtain $u = u^\varepsilon$ in $\mathbb{R}^N \times [0, T]$, which 
completes the proof.
\end{proof}

\begin{Prop} \label{TF_uni_space_regularity}
Assume (A1) - (A3) hold.
Let $u^\varepsilon$ be the viscosity solution of (\ref{TF_HJ_eqs}).
There exists a modulus of continuity $\omega \in C([0, \infty))$ independent of $\varepsilon$ such that
\begin{equation*}
|u^\varepsilon(x, t) - u^\varepsilon(y, t)| \leq \omega(|x - y|) \quad \mbox{for} \quad (x, y, t) \in \mathbb{R}^N \times \mathbb{R}^N \times [0, T].
\end{equation*}
Furthermore, assume (A5). 
There exists a constant 
$C > 0$ independent of $\varepsilon$ such that
\begin{equation*}
|u^\varepsilon(x, t) - u^\varepsilon(y, t)| \leq C(1+\big|\log|x-y|\big|)|x - y|
\end{equation*}
for $(x, y, t) \in \mathbb{R}^N \times \mathbb{R}^N \times [0, T]$.  
In particular, for each $\nu \in (0, 1)$, there exists a constant $C > 0$ independent of $\varepsilon$ such that
\begin{equation*}
|u^\varepsilon(x, t) - u^\varepsilon(y, t)| \leq C|x - y|^\nu \quad \mbox{for} \quad (x, y, t) \in \mathbb{R}^N \times \mathbb{R}^N \times [0, T].
\end{equation*}
\end{Prop}
We follow the proof of \cite[Theorem 2.2]{LTY} with a slight modification. 

\begin{Def}
Let $f : [0, T] \mapsto \mathbb{R}$ be a bounded function.
For each $\delta > 0$, we define $f^\delta$ and $f_\delta$ by
\begin{equation*}
f^\delta(t) \coloneqq \sup_{\xi \in [0, T]} \left\{ f(\xi) - \frac{|t - \xi|^2}{2\delta} \right\}, \quad f_\delta(t) \coloneqq \inf_{\xi \in [0, T]} \left\{ f(\xi) + \frac{|t - \xi|^2}{2\delta} \right\}
\end{equation*}
for $t \in [0, T]$.
We call $f^\delta$ and $f_\delta$ the sup-convolution and the inf-convolution, respectively.
\end{Def}

\begin{Lem} \label{sup_conv_property}
Let $f : [0, T] \mapsto \mathbb{R}$ be a bounded function.
There exists $C > 0$, for each $\delta > 0$,
\begin{enumerate}
\item [(a)] $f(t) \leq f^\delta(t) \leq \| f \|_{\infty}$ for $t \in [0, T]$,
\item [(b)] $\displaystyle |f^\delta(t) - f^\delta(s)| \leq \frac{C}{\delta} |t-s|$ for $t, s \in [0, T]$.
\end{enumerate}
Furthermore, assume that $f \in C^{0, \alpha}([0, T])$ for $\alpha \in (0, 1]$, then
\begin{enumerate}
\item [(c)] $|t - \xi_\delta|^{2-\alpha} \leq 2 M \delta $ for $\xi_\delta \in \argmax \{ f(\xi) - \frac{|t - \xi|^2}{2\delta} \mid \xi \in [0, T] \}$,
\item [(d)] $\| f^\delta - f \|_{\infty} \leq (2 M)^{\frac{2}{2 - \alpha}} \delta^{\frac{\alpha}{2-\alpha}}$,
\item [(e)] $f^\delta \in C^{0, \alpha}([0, T])$,
\end{enumerate}
where $M$ is the H\"{o}lder constant of $f$. 
\end{Lem}

\begin{proof}
We omit the proofs of (a) and (b), as they are standard.
We only prove (c), (d) and (e).
Take $\xi_\delta \in [0, T]$ so that $f^\delta(t) = f(\xi_\delta) - \frac{|t - \xi_\delta|^2}{2\delta}$.
Note that
\begin{equation} \label{TF_supconv_dist}
\frac{|t - \xi_\delta|^2}{2\delta} = f(\xi_\delta) - f^\delta(t) \leq f(\xi_\delta) - f(t) \leq M|t - \xi_\delta|^\alpha.
\end{equation}
Thus, $|t - \xi_\delta|^{2-\alpha} \leq 2 M \delta $.
By (\ref{TF_supconv_dist}), we have
\begin{align*}
|f^\delta(t) - f(t)| &= \left|f(\xi_\delta) - \frac{|t - \xi_\delta|^2}{2\delta} - f(t)\right| \\
&\leq |f(\xi_\delta) - f(t)| + \frac{|t - \xi_\delta|^2}{2\delta} \leq 2 M |t - \xi_\delta|^\alpha \leq (2 M)^{\frac{2}{2 - \alpha}} \delta^{\frac{\alpha}{2-\alpha}}
\end{align*}
for $t \in [0, T]$.
Finally, let $t, s \in [0, T]$.
We consider the case $|t - s|^{2 - \alpha} \geq \delta$.
By applying (d), we have
\begin{align*}
|f^\delta(t) - f^\delta(s)| &\leq |f^\delta(t) - f(t)| + |f(t) - f(s)| + |f(s) - f^\delta(s)| \\
&\leq 2 (2 M)^{\frac{2}{2 - \alpha}} \delta^{\frac{\alpha}{2 - \alpha}} + M |t - s|^\alpha \leq C |t - s|^\alpha.
\end{align*}
We consider the case $|t - s|^{2 - \alpha} \leq \delta$.
We have
\begin{align*}
f^\delta(t) - f^\delta(s) &\leq f(\xi_\delta) - \frac{|t - \xi_\delta|^2}{2\delta} - \left( f(\xi_\delta) - \frac{|s - \xi_\delta|^2}{2\delta} \right) \\
&= \frac{1}{2\delta} (|s - \xi_\delta|^2 - |t - \xi_\delta|^2) =  \frac{1}{2\delta} (|s - t + t - \xi_\delta|^2 - |t - \xi_\delta|^2) \\
&\leq \frac{1}{2\delta} (|t - s|^2 + 2|t - \xi_\delta||t - s|).
\end{align*}
By (c), we obtain
\begin{align*}
f^\delta(t) - f^\delta(s) &\leq \frac{1}{2\delta} (|t - s|^2 + 2 (2 M \delta)^{\frac{1}{2 - \alpha}} |t - s|)
\leq C\left( \frac{|t - s|^2}{\delta} + \frac{|t - s|}{\delta^{\frac{1 - \alpha}{2 - \alpha}}}\right) \\
&\leq C \left( \frac{|t - s|^2}{|t - s|^{2-\alpha}} + \frac{|t - s|}{|t - s|^{(2-\alpha)\frac{1 - \alpha}{2 - \alpha}}}\right) = 2 C |t - s|^\alpha, 
\end{align*}
which completes the proof.
\end{proof}

\begin{Lem}
Let $u$ be a bounded viscosity subsolution of \eqref{TF_HJ_eqs}, and for $\delta \in (0, 1)$, we write $u^\delta$ and $u_\delta$ for the sup-convolution and inf-convolution of $u$, respectively.
Then $u^\delta$ and $u_\delta$ are, respectively, a viscosity subsolution and a viscosity supersolution of
\begin{equation} \label{infsup_conv_eq}
\partial_t^\alpha u^{\delta} + H \left( \frac{x}{\varepsilon}, Du^{\delta} \right) = \eta_\delta \quad \mbox{in} \,\, \mathbb{R}^N \times (M\delta^{\frac{1}{4}} , T],
\end{equation}
where $M$ is a constant suct that $M^2 \geq 2\| u \|_\infty$ and $\eta_\delta$ is a constant such that $\eta_\delta \to 0$ as $\delta \to 0$. 
\end{Lem}
We refer to \cite[Lemma 4.2]{TY} and \cite[Proposition 4.2]{N} for similar results. 
\begin{proof}
We only prove that $u^\delta$ is a viscosity subsolution of \eqref{infsup_conv_eq}.
Take $((\hat{x}, \hat{t}), \varphi) \in (\mathbb{R}^N \times (M\delta^{\frac{1}{4}} , T]) \times (C^1(\mathbb{R}^N \times (0, T]) \cap C(\mathbb{R}^N \times [0, T]))$ so that 
\begin{equation*}
\max_{\mathbb{R}^N \times [M\delta^{\frac{1}{4}} , T]} (u^\delta - \varphi) = (u^\delta - \varphi)(\hat{x}, \hat{t}).
\end{equation*}
Let $\hat{t}_\delta \in [0, T]$ be such that
\begin{equation*}
u^\delta(\hat{x}, \hat{t}) = u(\hat{x}, \hat{t}_\delta) - \frac{ |\hat{t} - \hat{t}_\delta|^2 }{\delta}.
\end{equation*}
Then, we see that
\begin{equation*}
|\hat{t} - \hat{t}_\delta|^2 = (u(\hat{x}, \hat{t}_\delta) - u^\delta(\hat{x}, \hat{t})) \delta \leq 2 \| u \|_\infty \delta.
\end{equation*}
Thus, we have $\hat{t}_\delta \geq \hat{t} - (2 \| u \|_\infty \delta)^\frac{1}{2} > M\delta^\frac{1}{4} - (2 \| u \|_\infty \delta)^\frac{1}{2} \geq M\delta^\frac{1}{4}(1 - \delta^\frac{1}{4})> 0$.
Noting that $(u^\delta - \varphi)(x, t) \leq (u^\delta - \varphi)(\hat{x}, \hat{t})$, we see that
\begin{align*}
u(&\hat{x}, \hat{t}_\delta) - \frac{|\hat{t} - \hat{t}_\delta|^2}{\delta} - \varphi(\hat{x}, \hat{t}) = (u^\delta - \varphi)(\hat{x}, \hat{t})
\geq (u^\delta - \varphi)(x, t) \\
&= \sup_{\tau \in [0, T]} \left\{ u(x, \tau) - \frac{|t - \tau|^2}{\delta} \right\} - \varphi(x, t) 
\geq u(x, t + \hat{t}_\delta - \hat{t}) - \frac{|\hat{t} - \hat{t}_\delta|^2}{\delta} - \varphi(x, t)
\end{align*}
for $(x, t)$ in a neighborhood of $(\hat{x}, \hat{t})$.
Thus, the function $(x, t) \mapsto u(x, t) - \varphi(x, t - \hat{t}_\delta + \hat{t})$ attains a local maximum at $(\hat{x}, \hat{t}_\delta)$.
Therefore, there exists $\tilde{\varphi} \in C^1(\mathbb{R}^N \times (0, T]) \cap C(\mathbb{R}^N \times [0, T])$ such that $\tilde{\varphi}(x, t) = \varphi(x, t - \hat{t}_\delta + \hat{t})$ for $(x, t)$ in a neighborhood of $(\hat{x}, \hat{t}_\delta)$ and
\begin{equation*}
\max_{\mathbb{R}^N \times [0, T]} (u - \tilde{\varphi}) = (u - \tilde{\varphi})(\hat{x}, \hat{t}_\delta).
\end{equation*}
Since $u$ is a viscosity subsolution to \eqref{TF_HJ_eqs}, by Proposition \ref{vissol_equivalent}, we have
\begin{equation*}
J[u](\hat{x}, \hat{t}_\delta) + K_{(0, \hat{t}_\delta)}[u](\hat{x}, \hat{t}_\delta) + H \left( \frac{\hat{x}}{\varepsilon}, D\tilde{\varphi}(\hat{x}, \hat{t}_\delta) \right) \leq 0.
\end{equation*}

Define
\begin{equation*}
\tilde{u}(x, t) \coloneqq 
\left\{ 
\begin{array}{ll}
u(x, t) & \mbox{in} \, \, \mathbb{R}^N \times [0, T], \\
u(x, 0) & \mbox{in} \, \, \mathbb{R}^N \times (-\infty, 0].
\end{array} \right.
\end{equation*}
Then, we see that
\begin{align*}
J[u](\hat{x}, \hat{t}_\delta) + K_{(0, \hat{t}_\delta)}[u](\hat{x}, \hat{t}_\delta) &= \frac{\alpha}{\Gamma(1 - \alpha)} \int_{-\infty}^{\hat{t}_\delta} \frac{u(\hat{x}, \hat{t}_\delta) - \tilde{u}(\hat{x}, \tau)}{|\hat{t}_\delta - \tau|^{\alpha + 1}} d\tau \\
&= \frac{\alpha}{\Gamma(1 - \alpha)} \int_{-\infty}^{\hat{t}} \frac{u(\hat{x}, \hat{t}_\delta) - \tilde{u}(\hat{x}, \xi + \hat{t}_\delta - \hat{t})}{|\hat{t} - \xi|^{\alpha + 1}} d\xi \\
&= \frac{\alpha}{\Gamma(1 - \alpha)} \int_{-\infty}^{\hat{t}} \frac{u^\delta(\hat{x}, \hat{t}) + \frac{|\hat{t}_\delta - \hat{t}|^2}{\delta} - \tilde{u}(\hat{x}, \xi + \hat{t}_\delta - \hat{t})}{|\hat{t} - \xi|^{\alpha + 1}} d\xi.
\end{align*}
We set $\Psi(\xi) \coloneqq u^\delta(\hat{x}, \hat{t}) + \frac{|\hat{t}_\delta - \hat{t}|^2}{\delta} - \tilde{u}(\hat{x}, \xi + \hat{t}_\delta - \hat{t})$ for $\xi \in (-\infty, \hat{t}]$.
We divide into three cases: (1) $\xi \leq -M\delta^\frac{1}{2}$, (2) $-M\delta^\frac{1}{2} < \xi \leq M\delta^\frac{1}{2}$, and (3) $M\delta^\frac{1}{2} < \xi \leq \hat{t}$.
First, we consider case (1).
Since $\xi + \hat{t}_\delta - \hat{t} \leq -M\delta^\frac{1}{2} + M\delta^\frac{1}{2} = 0$, we have $\tilde{u}(\hat{x}, \xi + \hat{t}_\delta - \hat{t}) = u(\hat{x}, 0) \leq u^\delta(\hat{x}, 0) = \tilde{u}^\delta(\hat{x}, \xi)$.
We obtain
\begin{equation*}
\Psi(\xi) \geq u^\delta(\hat{x}, \hat{t}) + \frac{(t_\delta - \hat{t})^2}{\delta} - u^\delta(\hat{x}, 0) \geq u^\delta(\hat{x}, \hat{t}) - \tilde{u}^\delta(\hat{x}, \xi).
\end{equation*}
Next, in case (2), we have
\begin{align*}
\Psi(\xi) &\geq u^\delta(\hat{x}, \hat{t}) - \tilde{u}(\hat{x}, \xi + \hat{t}_\delta - \hat{t}) \\
&\geq u^\delta(\hat{x}, \hat{t}) - \tilde{u}^\delta(\hat{x}, \xi) + \tilde{u}^\delta(\hat{x}, \xi) - \tilde{u}(\hat{x}, \xi + \hat{t}_\delta - \hat{t}) \\
&\geq u^\delta(\hat{x}, \hat{t}) - \tilde{u}^\delta(\hat{x}, \xi) - 2\| u \|_\infty.
\end{align*}
Finally. we consider case (3).
Since $\xi + \hat{t}_\delta - \hat{t} > M\delta^\frac{1}{2} - M\delta^\frac{1}{2} = 0$, we have $\tilde{u}(\hat{x}, \xi + \hat{t}_\delta - \hat{t}) = u(\hat{x}, \xi + \hat{t}_\delta - \hat{t})$.
Therefore, we obtain
\begin{align*}
\Psi(\xi) &= u^\delta(\hat{x}, \hat{t}) + \frac{|\hat{t}_\delta - \hat{t}|^2}{\delta} - u(\hat{x}, \xi + \hat{t}_\delta - \hat{t}) \\
&= u^\delta(\hat{x}, \hat{t}) + \frac{|\xi - (\xi + \hat{t}_\delta - \hat{t})|^2}{\delta} - \tilde{u}(\hat{x}, \xi + \hat{t}_\delta - \hat{t})
\geq u^\delta(\hat{x}, \hat{t}) - \tilde{u}^\delta(\hat{x}, \xi).
\end{align*}
Hence, we obtain
\begin{align*}
J[u](\hat{x}, \hat{t}_\delta) + K_{(0, \hat{t}_\delta)}[u](\hat{x}, \hat{t}_\delta) &= \frac{\alpha}{\Gamma(1 - \alpha)} \int_{-\infty}^{\hat{t}} \frac{\Psi(\xi)}{|\hat{t} - \xi|^{\alpha + 1}} d\xi \\
&\geq \frac{\alpha}{\Gamma(1 - \alpha)} \left( \int_{-\infty}^{\hat{t}} \frac{u^\delta(\hat{x}, \hat{t}) - \tilde{u}^\delta(\hat{x}, \xi)}{|\hat{t} - \xi|^{\alpha + 1}} d\xi - \int_{-M\delta^\frac{1}{2}}^{M\delta^\frac{1}{2}} \frac{2 \|u\|_\infty}{|\hat{t} - \xi|^{\alpha + 1}} d\xi \right) \\
&\geq J[u^\delta](\hat{x}, \hat{t}) + K_{(0, \hat{t})}[u^\delta](\hat{x}, \hat{t}_\delta) - \frac{\alpha M^2}{\Gamma(1 - \alpha)} \int_{-M\delta^\frac{1}{2}}^{M\delta^\frac{1}{2}} \frac{1}{|\hat{t} - \xi|^{\alpha + 1}} d\xi.
\end{align*}
Noting that $\hat{t} > M\delta^\frac{1}{4}$ and $\frac{1}{2} \delta^{\frac{1}{4}} > \delta^{\frac{1}{2}}$ for small $0 < \delta < 1$, we have
\begin{align*}
\int_{-M\delta^\frac{1}{2}}^{M\delta^{\frac{1}{2}}} \frac{1}{|\hat{t} - \xi|^{\alpha + 1}} d\xi
&\leq \int_{-M\delta^{\frac{1}{2}}}^{M\delta^{\frac{1}{2}}} \frac{1}{(M\delta^\frac{1}{4} - M\delta^\frac{1}{2})^{\alpha + 1}} d\xi \\
&\leq \left( \frac{M\delta^{\frac{1}{4}}}{2} \right)^{-(\alpha+1)} 2M\delta^\frac{1}{2} = 2^{\alpha+2} M^{-\alpha} \delta^{\frac{1-\alpha}{4}},
\end{align*}
which implies
\begin{equation*}
J[u](\hat{x}, \hat{t}_\delta) + K_{(0, \hat{t}_\delta)}[u](\hat{x}, \hat{t}_\delta)
\geq
J[u^\delta](\hat{x}, \hat{t}) + K_{(0, \hat{t})}[u^\delta](\hat{x}, \hat{t}_\delta) - \eta_\delta,
\end{equation*}
where
\begin{equation*}
\eta_\delta \coloneqq \frac{2^{\alpha+2} \alpha M^{2-\alpha}}{\Gamma(1 - \alpha)} \delta^{\frac{1-\alpha}{4}}.
\end{equation*}
This completes the proof.
\end{proof}

\begin{proof} [Proof of Proposition \ref{TF_uni_space_regularity}]

For $\delta > 0$, we denote by $u^{\varepsilon, \delta}$ the sup-convolution of $u^\varepsilon$.
For fixed $y \in \mathbb{R}^N, t_0 > 0,$ and $\beta, L > 0$, we set
\begin{equation*}
\Phi(x, t) \coloneqq u^{\varepsilon, \delta}(x, t) - u^{\varepsilon, \delta}(y, t_0) - L|x - y| - \frac{|t - t_0|^2}{\beta^2},
\end{equation*}
where $L$ will be fixed later.
Since $\Phi(x, t) \to -\infty$ as $|x| \to \infty$ for any $t \in [0, T]$ and $\Phi$ is bounded from above, $\Phi$ attains a maximum at a point $(\hat{x}, \hat{t}) \in \mathbb{R}^N \times [0, T]$.
We claim that if $L$ is sufficiently large, then $\hat{x} = y$.
Suppose that $\hat{x} \neq y$.
Since $u^{\varepsilon, \delta}$ is the viscosity subsolution of (\ref{TF_HJ_eqs}), by Proposition \ref{vissol_equivalent}, for every $0 < \rho < \hat{t}$, we have
\begin{equation} \label{TF_space_reg_aprox}
J[u^{\varepsilon, \delta}](\hat{x}, \hat{t}) + K_{(0, \rho)}[u^{\varepsilon, \delta}](\hat{x}, \hat{t}) +  K_{(\rho, \hat{t})}[u^{\varepsilon, \delta}](\hat{x}, \hat{t}) + H \left (\frac{\hat{x}}{\varepsilon}, L\frac{\hat{x} - y}{|\hat{x} - y|} \right) \leq \eta_\delta.
\end{equation}
By Lemma \ref{sup_conv_property} (e), we have
\begin{equation*}
J[u^{\varepsilon, \delta}](\hat{x}, \hat{t}) = \frac{u^{\varepsilon, \delta}(\hat{x}, \hat{t}) - u^{\varepsilon, \delta}(\hat{x}, 0)}{\hat{t}^\alpha \Gamma(1 - \alpha)} \geq 
- \frac{C}{\Gamma(1 - \alpha)},
\end{equation*}
and
\begin{align*}
K_{(\rho, \hat{t})}[u^{\varepsilon, \delta}](\hat{x}, \hat{t}) &= \frac{\alpha}{\Gamma(1 - \alpha)} \int_\rho^{\hat{t}} \frac{u^{\varepsilon, \delta}(\hat{x}, \hat{t}) - u^{\varepsilon, \delta}(\hat{x}, \hat{t} - \tau)}{\tau^{\alpha + 1}} d\tau \\
&\geq \frac{\alpha}{\Gamma(1 - \alpha)} \int_{\rho}^{\hat{t}} \frac{- C\tau^\alpha}{\tau^{\alpha + 1}} d\tau = - \frac{\alpha C}{\Gamma(1 - \alpha)} \log \frac{\hat{t}}{\rho}.
\end{align*}{}
Next, by Lemma \ref{sup_conv_property} (b), we have
\begin{align*}
K_{(0, \rho)}[u^{\varepsilon, \delta}](\hat{x}, \hat{t}) &= \frac{\alpha}{\Gamma(1 - \alpha)} \int_0^\rho \frac{u^{\varepsilon, \delta}(\hat{x}, \hat{t}) - u^{\varepsilon, \delta}(\hat{x}, \hat{t} - \tau)}{\tau^{\alpha + 1}} d\tau \\
&\geq \frac{\alpha}{\Gamma(1 - \alpha)} \int_0^\rho \frac{-\frac{C}{\delta} \tau}{\tau^{\alpha + 1}} d\tau
= - \frac{\alpha}{\Gamma(1 - \alpha)} \frac{C}{\delta} \frac{\rho^{1 - \alpha}}{1 - \alpha}.
\end{align*}
Combining these with (\ref{TF_space_reg_aprox}), we obtain
\begin{equation*}
H \left( \frac{\hat{x}}{\varepsilon}, L\frac{\hat{x} - y}{|\hat{x} - y|} \right) \leq \eta_\delta + \frac{C}{\Gamma(1 - \alpha)} + \frac{\alpha C}{\Gamma(1 - \alpha)} \log \frac{\hat{t}}{\rho} + \frac{\alpha}{\Gamma(1 - \alpha)} \frac{C}{\delta} \frac{\rho^{1 - \alpha}}{1 - \alpha} \leq C (1 + \frac{\rho^{1-\alpha}}{\delta} + \log \frac{\hat{t}}{\rho}).
\end{equation*}
By setting $\rho = \hat{t}\delta^{\frac{1}{1-\alpha}}$, we obtain
\begin{equation} \label{Sp_reg_coer_esti}
H \left( \frac{\hat{x}}{\varepsilon}, L\frac{\hat{x} - y}{|\hat{x} - y|} \right) \leq C (1 + \hat{t}^{1-\alpha} + \frac{1}{1-\alpha} |\log \delta|) \leq C (1 + |\log \delta|). 
\end{equation}
If we take $L = L(\delta)$ large, then, by the coercivity of $H$, we get a contradiction.

Therefore, we have $\Phi(x, \hat{t}) \leq \Phi(\hat{x}, \hat{t}) = \Phi(y, \hat{t})$, i.e.,
\begin{equation*}
u^{\varepsilon, \delta}(x, \hat{t}) - u^{\varepsilon, \delta}(y, t_0) - L(\delta)|x - y| - \frac{|\hat{t} - t_0|^2}{\beta^2}
\leq
u^{\varepsilon, \delta}(y, \hat{t}) - u^{\varepsilon, \delta}(y, t_0) - \frac{|\hat{t} - t_0|^2}{\beta^2}.
\end{equation*}
Taking $\beta \to 0$, we have
\begin{equation*}
u^{\varepsilon, \delta}(x, t_0) - u^{\varepsilon, \delta}(y, t_0) \leq L(\delta) |x - y| \quad \mbox{for all} \,\, x \in \mathbb{R}^N.
\end{equation*}
Therefore, by Lemma \ref{sup_conv_property} (d), we obtain
\begin{equation*}
u^\varepsilon(x, t) - u^\varepsilon(y, t) \leq u^{\varepsilon, \delta}(x, t) - u^{\varepsilon, \delta}(y, t) + C \delta^{\frac{\alpha}{2 - \alpha}} 
\leq L(\delta)|x - y| + C \delta^{\frac{\alpha}{2 - \alpha}}
\end{equation*}
for some $C > 0$.

Finally, define
\begin{equation*}
\omega(r) \coloneqq \inf_{\delta > 0} \{ L(\delta)r + C \delta^{\frac{\alpha}{2 - \alpha}} \}.
\end{equation*}
We claim that $\omega(r) \to 0$ as $r \to 0$.
Letting $\delta_r > 0$ so that $L(\delta_r) = r^{-\frac{1}{2}}$, we have
\begin{equation*}
0 \leq \omega(r) \leq r^{\frac{1}{2}} + C\delta_r^{\frac{\alpha}{2 - \alpha}}.
\end{equation*}
Since $\delta_r \to 0$ as $r \to 0$, taking $r \to 0$ as above, this completes the proof of modulus of continuity.

In addition, we assume (A5), i.e., there exists a constant $C > 0$ such that
\begin{equation*}
H(x, p) \geq C^{-1} |p| - C \quad \mbox{for} \, \, x, p \in \mathbb{R}^N.
\end{equation*}
By \eqref{Sp_reg_coer_esti}, we have
\begin{equation*}
L \leq C(1 + |\log \delta|).
\end{equation*}
By similar argument to the above, we obtain
\begin{equation}
u^\varepsilon(x, t) - u^\varepsilon(y, t) \leq C (1 + |\log \delta|) |x - y| + C\delta^{\frac{\alpha}{2 - \alpha}}
\end{equation}
for $x, y \in \mathbb{R}^N$ and $t \in (0, T]$.
Thus, taking the infimum with respect to $\delta$, we have
\begin{align*}
u^\varepsilon(x, t) - u^\varepsilon(y, t) &\leq \inf_{\delta  > 0} \{ C (1 + |\log \delta|) |x - y| + C\delta^{\frac{\alpha}{2 - \alpha}} \} \\
&\leq C (1 + |\log |x - y|^{\frac{2-\alpha}{\alpha}} |) |x - y| + C|x - y|^{\frac{2-\alpha}{\alpha} \cdot \frac{\alpha}{2 - \alpha}} \\
&\leq C \left( 1 + \left| \log |x - y| \right| \right) |x - y|,
\end{align*}
which completes the proof.
\end{proof}

\begin{Rem}
In \cite[Lemma 6.1]{GN}, they show that Lipschitz continuous with respect to the space variable.
However, applying this result directly to our problem, we obtain
\begin{equation*}
|u^\varepsilon(x, t) - u^\varepsilon(y, t)| \leq \left( \Lip[u_0] + \frac{Ct^\alpha}{\ep}\right) |x - y|.
\end{equation*}
Therefore, this does not give us a uniform Lipschitz estimate.
On the other hand, by \cite[Lemma 6.1]{GN} we see that the limit problem \eqref{TF_HJ} has the unique Lipschitz continuous viscosity solution.
Therefore, we can expect a $\ep$-uniform Lipschitz estimate, but this still remains open.
\end{Rem}

\section{Perturbed test function method}


\begin{Lem} \label{TF_half_relaxed_lim}
For each $\varepsilon > 0$, let $u^\varepsilon$ be the viscosity solution of (\ref{TF_HJ_eqs}). Set
\begin{align*}
\displaystyle & u^*(x, t) \coloneqq \limsup_{\varepsilon \to 0} {}^* u^\varepsilon(x, t) = \lim_{\varepsilon \to 0} \sup \{ u^\delta (y, s) \mid |x - y| \leq \varepsilon, |t - s| \leq \varepsilon, \delta \leq \varepsilon \}, \\
\displaystyle & u_*(x, t) \coloneqq \liminf_{\varepsilon \to 0} {}_* u^\varepsilon(x, t) = \lim_{\varepsilon \to 0} \inf \{ u^\delta (y, s) \mid |x - y| \leq \varepsilon, |t - s| \leq \varepsilon, \delta \leq \varepsilon \}.
\end{align*}
Then, $u^*$ and $u_*$ are a viscosity subsolution and a viscosity supersolution of (\ref{TF_HJ}), respectively.
\end{Lem}

\begin{proof}
We only prove that $u^*$ is a viscosity subsolution of (\ref{TF_HJ}), we can similarly prove that $u_*$ is a viscosity supersolution of \eqref{TF_HJ}.

Let $\varphi \in C^1(\mathbb{R}^N \times (0, T]) \cap C(\mathbb{R}^N \times [0, T])$ such that $u^* - \varphi$ has a strict maximum at $(\hat{x}, \hat{t}) \in \mathbb{R}^N \times (0, T]$.
Let $P = D\varphi(\hat{x}, \hat{t}) \in \mathbb{R}^N$, and let $v \in $ Lip $(\mathbb{T}^N)$ be a viscosity solution of (\ref{TF_CP}).
Choose a sequence $\varepsilon_m \to 0$ $(m \in \mathbb{N})$ such that
\begin{equation*}
\displaystyle u^*(\hat{x}, \hat{t}) = \lim_{m \to \infty} u^{\varepsilon_m}(x_m, t_m) \quad \mbox{for} \, (x_m, t_m) \in \mathbb{R}^N \times (0, T).
\end{equation*}

We set
\begin{equation*}
\Phi^{m, a}(x, y, t) \coloneqq u^{\varepsilon_m}(x, t) - \varphi(x, t) - \varepsilon_m v \left( \frac{y}{\varepsilon_m} \right) - \frac{|x - y|^2}{2a} \quad \mbox{for} \, a > 0.
\end{equation*}
For every $m \in \mathbb{N}$ and $a > 0$, there exists $(x_{m, a}, y_{m, a}, t_{m, a}) \in \mathbb{R}^N \times \mathbb{R}^N \times (0, T)$ such that
\begin{equation*}
\max_{\mathbb{R}^N \times \mathbb{R}^N \times [0, T]} \Phi^{m, a}(x, y, t) = \Phi^{m, a} (x_{m, a}, y_{m, a}, t_{m, a})
\end{equation*}
and up to passing some subsequences
\begin{equation*}
\left \{
\begin{array}{ll}
(x_{m, a}, y_{m, a}, t_{m, a}) \to (x_m, x_m, t_m) & \mbox{as} \quad a \to 0, \\
(x_m, t_m) \to (\hat{x}, \hat{t}) & \mbox{as} \quad m \to \infty, \\
\displaystyle \lim_{m \to \infty} \lim_{a \to \infty} u^{\varepsilon_m}(x_{m, a}, t_{m, a}) = u^*(\hat{x}, \hat{t}).
\end{array}
\right.
\end{equation*}

Notice that the function $(x, t) \mapsto \Phi^{m, a}(x, y_{m, a}, t)$ has a maximum at $(x_{m, a}, t_{m, a})$.
Since $u^{\varepsilon_m}$ is a viscosity subsolution of (\ref{TF_HJ_eqs}), we have
\begin{equation} \label{Hom_u^ep}
\displaystyle J[\varphi](x_{m, a}, t_{m, a}) + K_{(0, t_{m, a})}[\varphi](x_{m, a}, t_{m, a}) + H \left( \frac{x_{m, a}}{\varepsilon_m}, D\varphi(x_{m, a}, t_{m, a}) + \frac{x_{m, a} - y_{m, a}}{a} \right) \ \leq  0.
\end{equation}
On the other hand, notice that the function $y \mapsto -\frac{1}{\varepsilon_m} \Phi^{m, a}(x_{m, a}, \varepsilon_m y, t_{m, a})$ has a minimum at $\frac{y_{m, a}}{\varepsilon_m}$.
Since $v$ is a viscosity supersolution of (\ref{TF_CP}), we have
\begin{equation} \label{Hom_v}
\displaystyle H \left( \frac{y_{m, a}}{\varepsilon_m}, P + \frac{x_{m, a} - y_{m, a}}{a} \right) \ \geq  \overline{H}(P).
\end{equation}
Since $\Phi^{m, a}(x_{m, a}, y_{m, a}, t_{m, a}) \geq \Phi^{m, a}(x_{m, a}, x_{m, a}, t_{m, a})$, we have
\begin{equation*}
\displaystyle \frac{|x_{m, a} - y_{m, a}|^2}{2a} \leq \varepsilon_m ( v(\frac{x_{m, a}}{\varepsilon_m}) - v(\frac{y_{m, a}}{\varepsilon_m}) ) \leq \varepsilon_m L \frac{|x_{m, a} - y_{m, a}|}{\varepsilon_m} = L |x_{m, a} - y_{m, a}|,
\end{equation*}
where $L$ is a Lipschitz constant of $v$.
Thus, we have $\frac{|x_{m, a} - y_{m, a}|}{a} \leq 2L$.
If necessary, taking a subsequence, we can assume that
\begin{equation*}
\frac{x_{m, a} - y_{m, a}}{a} \to Q_m \in \mathbb{R}^N \quad \mbox{as} \, \, a \to 0.
\end{equation*}

By Proposition \ref{K_conti} and taking the limit $a \to 0$ in (\ref{Hom_u^ep}) and (\ref{Hom_v}) to get
\begin{align*}
\displaystyle J[\varphi](x_{m}, t_{m}) + K_{(0, t_{m})}[\varphi](x_{m}, t_{m}) + H \left( \frac{x_m}{\varepsilon_m}, D\varphi(x_m, t_m) + Q_m \right) \ \leq  0, \\
\displaystyle H \left( \frac{x_m}{\varepsilon_m}, P + Q_m \right) \geq \overline{H}(P).
\end{align*}
Combine the above two inequality, we obtain
\begin{align*}
\displaystyle &J[\varphi](x_{m}, t_{m}) + K_{(0, t_{m})}[\varphi](x_{m}, t_{m}) + \overline{H}(P) \\
& \leq H \left( \frac{x_m}{\varepsilon_m}, P + Q_m \right) - H \left( \frac{x_m}{\varepsilon_m}, D\varphi(x_m, t_m) + Q_m \right) \\
& \leq \omega_H (|P - D\varphi(x_m, t_m)|).
\end{align*}
Taking the limit $m \to \infty$ to conclude that
\begin{equation*}
J[\varphi](\hat{x}, \hat{t}) + K_{(0, \hat{t})}[\varphi](\hat{x}, \hat{t}) + \overline{H} ( D\varphi(\hat{x}, \hat{t}) ) \ \leq  0.
\end{equation*} 
\end{proof}

\begin{Lem} \label{TF_half_relaxed_uni}
Let $u^*$ and $u_*$ be the functions defined by Lemma \ref{TF_half_relaxed_lim}.
Then, $u^* = u_* = \overline{u}$ in $\mathbb{R}^N \times [0, T]$,
where $\overline{u}$ is the unique viscosity solution to (\ref{TF_HJ}).
\end{Lem}

\begin{proof}
By definition of half relaxed limit, we have $u_* \leq u^*$ in $\mathbb{R}^N \times [0, T]$.
On the other hand, by the comparison principle for \eqref{TF_HJ}, we have $u^* \leq u_*$ in $\mathbb{R}^N \times [0, T]$, which implies the conclusion.
\end{proof}

Theorem \ref{TF_conv_Main_Theorem_0} is a straightforward result of Lemmas \ref{TF_half_relaxed_lim} and \ref{TF_half_relaxed_uni}

\section{Rate of convergence}

Let us consider the discounted approximation for cell problem \eqref{TF_CP}:
\begin{equation} \label{TF_ACP_2}
\lambda v^\lambda + H(y, P + Dv^\lambda) = 0 \quad \mbox{in} \, \, \mathbb{R}^N.
\end{equation}
We recall some results of $v^\lambda$, and refer to \cite[Lemma 2.3]{CDI} and \cite[Lemma 4.38]{T} for proofs.

\begin{Lem} \label{TF_ACP_Property}
\begin{enumerate}
\item [(a)] There exists a constant $C > 0$ independent of $\lambda > 0$ such that
\begin{equation*}
\lambda |v^\lambda(y, p) - v^\lambda(y, q)| \leq C|p - q| \quad \mbox{for} \, \, y, p, q \in \mathbb{R}^N.
\end{equation*}
In particular, $|\overline{H}(p) - \overline{H}(q)| \leq C|p - q|$ for $p, q \in \mathbb{R}^N$.

\item [(b)] For any $p \in \mathbb{R}^N$, there exists a constant $C > 0$ independent of $\lambda > 0$ and $p$ such that
\begin{equation*}
|\lambda v^\lambda(y, p) + \overline{H}(p)| \leq C (1 + |p|) \lambda \quad \mbox{for} \, \, y \in \mathbb{R}^N.
\end{equation*}
\end{enumerate}
\end{Lem}

The next lemma is a key ingredient to prove Theorem \ref{TF_conv_Main_Theorem}.
\begin{Lem} \label{TF_Conv_Lem}
For $\varepsilon, \lambda > 0$, let $u^\varepsilon$, $\overline{u}$ and $v^\lambda$ be the viscosity solutions of (\ref{TF_HJ_eqs}), (\ref{TF_HJ}) and (\ref{TF_ACP_2}), respectively.
Let $\Phi : \mathbb{R}^{2N}\times [0, T] \mapsto \mathbb{R}$ be the function defined by
\begin{equation*}
\Phi(x, y, t) \coloneqq u^\varepsilon(x, t) - \overline{u}(y, t) - \varepsilon v^{\lambda} \left( \frac{x}{\varepsilon}, \frac{x - y}{\varepsilon^\beta} \right) - \frac{|x-y|^2}{2\varepsilon^\beta} - \frac{\delta|y|^2}{2} - \frac{Rt^\alpha}{\Gamma(1 + \alpha)},
\end{equation*}
where $\lambda = \varepsilon^\theta$ and $\theta, \beta, \delta \in (0, 1),$ and $R > 0$.
Assume that $\Phi(x, y, t)$ attains a maximum at $(\hat{x}, \hat{y}, \hat{t}) \in \mathbb{R}^{2N} \times (0, T]$ over $\mathbb{R}^{2N} \times[0, T] $ and $\delta^{\frac{1}{2}} < \lambda$.
There exists $\varepsilon_0 \in (0, 1)$ such that for each $\nu \in (0, 1)$, there exists $C > 0$ which is independent of $\varepsilon$ such that
\begin{equation*}
J[u^\varepsilon](\hat{x}, \hat{t}) - J[\overline{u}](\hat{y}, \hat{t})
+ K_{(0, \hat{t})}[u^\varepsilon](\hat{x}, \hat{t}) - K_{(0, \hat{t})}[\overline{u}](\hat{y}, \hat{t})
\leq C (\varepsilon^{\theta - \frac{\beta(1-\nu)}{2-\nu}} + \varepsilon^{1-\theta-\beta})
\end{equation*}
for all $\varepsilon \in (0, \varepsilon_0)$.
\end{Lem}

\begin{proof}
STEP 1. 
We claim that if $1 - \theta - \beta > 0$, then there exists $C > 0$ such that 
\begin{equation} \label{TF_Conv_y}
|\hat{y}| \leq C \delta^{-\frac{1}{2}}.
\end{equation}
Additionally, if $\delta < \frac{1}{2}$, then
\begin{equation} \label{TF_Conv_diff_xy_pre}
|\hat{x} - \hat{y}| \leq C \varepsilon^{\frac{\beta}{2 - \nu}}.
\end{equation}

Noting that $\Phi(0, 0, \hat{t}) \leq \Phi(\hat{x}, \hat{y}, \hat{t})$, i.e.,
\begin{equation*}
u^\varepsilon(0, \hat{t}) - \overline{u}(0, \hat{t}) - \varepsilon v^{\lambda}(0, 0)
\leq
u^\varepsilon(\hat{x}, \hat{t}) - \overline{u}(\hat{y}, \hat{t}) - \varepsilon v^{\lambda} \left( \frac{\hat{x}}{\varepsilon}, \frac{\hat{x} - \hat{y}}{\varepsilon^\beta} \right) - \frac{|\hat{x}-\hat{y}|^2}{2\varepsilon^\beta} - \frac{\delta |\hat{y}|^2}{2},
\end{equation*}
which implies
\begin{align*}
\frac{\delta |\hat{y}|^2}{2} + \frac{|\hat{x}-\hat{y}|^2}{2\varepsilon^\beta}
&\leq
u^\varepsilon(\hat{x}, \hat{t}) - u^\varepsilon(0, \hat{t})
- (\overline{u}(\hat{y}, \hat{t}) - \overline{u}(0, \hat{t}) ) \\
& \quad - \varepsilon \left\{ v^{\lambda} \left( \frac{\hat{x}}{\varepsilon}, \frac{\hat{x} - \hat{y}}{\varepsilon^\beta} \right) - v^{\lambda}(0, 0) \right\}.
\end{align*}
Noting that, by Lemma \ref{TF_ACP_Property},
\begin{align*}
\left| v^{\lambda} \left( \frac{\hat{x}}{\varepsilon}, \frac{\hat{x} - \hat{y}}{\varepsilon^\beta} \right) - v^{\lambda}(0, 0) \right|
&\leq
\left| v^{\lambda} \left( \frac{\hat{x}}{\varepsilon}, \frac{\hat{x} - \hat{y}}{\varepsilon^\beta} \right) - v^{\lambda} \left( \frac{\hat{x}}{\varepsilon}, 0 \right) \right|
+ \left| v^{\lambda} \left( \frac{\hat{x}}{\varepsilon}, 0 \right) \right|
+ |v^{\lambda}(0, 0)| \\
&\leq
\frac{C}{\lambda} \frac{|\hat{x} - \hat{y}|}{\varepsilon^\beta} + \frac{2C}{\lambda} (|\overline{H}(0)| + \lambda),
\end{align*}
we obtain
\begin{align*}
\frac{\delta |\hat{y}|^2}{2} + \frac{|\hat{x} - \hat{y}|^2}{2\varepsilon^\beta}
&\leq
2 \| u^\varepsilon \|_\infty + 2 \| \overline{u} \|_\infty + \varepsilon \left| v^{\lambda} \left( \frac{\hat{x}}{\varepsilon}, \frac{\hat{x} - \hat{y}}{\varepsilon^\beta} \right) - v^{\lambda}(0, 0) \right| \\
&\leq C + \frac{C \varepsilon}{\lambda} \left( \frac{|\hat{x} - \hat{y}|}{\varepsilon^\beta} + 1 \right) = C + \frac{C \varepsilon^{1 - \frac{\beta}{2}} |\hat{x} - \hat{y}|}{\lambda \varepsilon^{\frac{\beta}{2}}} + \frac{C \varepsilon}{\lambda}.
\end{align*}
With $\lambda = \varepsilon^\theta$, by using the Young inequality, we obtain
\begin{equation*}
\frac{\delta |\hat{y}|^2}{2} + \frac{|\hat{x}-\hat{y}|^2}{2\varepsilon^\beta}
\leq
C + C \varepsilon^{2(1 - \theta - \frac{\beta}{2})} + \frac{|\hat{x}-\hat{y}|^2}{2\varepsilon^\beta} + C \varepsilon^{1-\theta},
\end{equation*}
which implies \eqref{TF_Conv_y}.

Next, by $\Phi(\hat{x}, \hat{x}, \hat{t}) \leq \Phi(\hat{x}, \hat{y}, \hat{t})$, i.e.,
\begin{equation*}
- \overline{u}(\hat{x}, \hat{t}) - \varepsilon v^{\lambda}\left(\frac{\hat{x}}{\varepsilon}, 0\right) - \frac{\delta |\hat{x}|^2}{2}
\leq
- \overline{u}(\hat{y}, \hat{t}) - \varepsilon v^{\lambda}\left(\frac{\hat{x}}{\varepsilon}, \frac{\hat{x} - \hat{y}}{\varepsilon^\beta}\right) - \frac{|\hat{x} - \hat{y}|^2}{2 \varepsilon^\beta} - \frac{\delta |\hat{y}|^2}{2},
\end{equation*}
which implies, by Proposition \ref{TF_uni_space_regularity}, Lemma \ref{TF_ACP_Property}, for $\nu \in (0, 1)$, we obtain
\begin{align}
\frac{|\hat{x} - \hat{y}|^2}{2 \varepsilon^\beta} &\leq \overline{u}(\hat{x}, \hat{t}) - \overline{u}(\hat{y}, \hat{t}) + \varepsilon \left\{ v^{\lambda}\left(\frac{\hat{x}}{\varepsilon}, 0\right) - v^{\lambda}\left(\frac{\hat{x}}{\varepsilon}, \frac{\hat{x} - \hat{y}}{\varepsilon^\beta}\right) \right\} + \delta \frac{|\hat{x}|^2 - |\hat{y}|^2}{2} \notag \\
&\leq C|\hat{x} - \hat{y}|^\nu + \varepsilon \frac{C}{\lambda} \frac{|\hat{x} - \hat{y}|}{\varepsilon^\beta} + \frac{\delta}{2} (|\hat{x} - \hat{y}|^2 + 2 |\hat{x} - \hat{y}| |\hat{y}|).  \label{Reglarity_modify}
\end{align}
It is suffices to consider case $|\hat{x} - \hat{y}| < 1$ 
if $\ep$ is small enough. 
Combining \eqref{Reglarity_modify} with \eqref{TF_Conv_y}, we have
\begin{align*}
(1 - \delta \varepsilon^\beta) \frac{|\hat{x} - \hat{y}|^2}{2 \varepsilon^\beta}
&\leq
C ( 1 + \varepsilon^{1 - \theta - \beta}|\hat{x} - \hat{y}|^{1-\nu} + \delta^{\frac{1}{2}}|\hat{x} - \hat{y}|^{1-\nu} ) |\hat{x} - \hat{y}|^\nu \\
&\leq
C ( 1 + \varepsilon^{1 - \theta - \beta} + \delta^{\frac{1}{2}}) |\hat{x} - \hat{y}|^\nu.
\end{align*}
Since $\delta < \frac{1}{2}$ and $0 < 1 - \theta - \beta$, we get (\ref{TF_Conv_diff_xy_pre}).

\medskip
STEP 2. 
For $\gamma > 0$, we consider the auxiliary function $\Psi:\mathbb{R}^{4N}\times [0, T]^2 \mapsto \mathbb{R}$ defined by
\begin{equation*}
\begin{split}
\Psi(x, y, z, \xi, t, s) \coloneqq \, & u^\varepsilon(x, t) - \overline{u}(y, s) - \varepsilon v^\lambda \left( \xi, \frac{z}{\varepsilon^\beta} \right) \\
& - \frac{|x-y|^2}{2\varepsilon^\beta} - \frac{\delta |y|^2}{2} - \frac{R t^\alpha}{\Gamma(1 + \alpha)} \\
& - \frac{|x - \varepsilon \xi|^2 + |x-y-z|^2 + |t - s|^2}{2\gamma} \\
& - \frac{|x - \hat{x}|^2 + |y - \hat{y}|^2 + |t - \hat{t}|^2}{2}.
\end{split}
\end{equation*}
Since $\Psi(x, y, z, \xi, t, s) \to -\infty$ as $|x|, |y|, |z|, |\xi| \to \infty$ for any $(t, s) \in [0, T]^2$ and $\Psi$ is bounded from above, $\Psi$ attains a maximum at a point $(x_\gamma, y_\gamma, z_\gamma, \xi_\gamma, t_\gamma, s_\gamma) \in \mathbb{R}^{4N}\times [0, T]^2$.
By the standard argument of the theory of viscosity solutions, we obtain
\begin{equation*}
(x_\gamma, y_\gamma, z_\gamma, \xi_\gamma, t_\gamma, s_\gamma) \to \left( \hat{x}, \hat{y}, \hat{x} - \hat{y}, \frac{\hat{x}}{\varepsilon}, \hat{t}, \hat{t} \right) \quad \mbox{as} \,\, \gamma \to 0.
\end{equation*}
If necessary, taking $\gamma$ sufficiently small, we can assume that $t_\gamma, s_\gamma > 0$. 
Our goal of STEP 2 is to obtain
\begin{equation} \label{TF_Conv_STEP2}
\begin{split}
J[u^\varepsilon](x_\gamma, t_\gamma) - J[\overline{u}](y_\gamma, s_\gamma)
+ K_{(0, t_\gamma)}[u^\varepsilon] &(x_\gamma, t_\gamma) - K_{(0, s_\gamma)}[\overline{u}](y_\gamma, s_\gamma) \\
&\leq C (|z_\gamma|\varepsilon^{\theta-\beta} + \varepsilon^\theta + \varepsilon^{1-\theta-\beta}) + E_\gamma,
\end{split}
\end{equation}
where $E_\gamma$ is a constant which satisfies $E_\gamma \to 0$ as $\gamma \to 0$, and will be defined later.

\medskip
STEP 2-1. 
We claim that there exists $C > 0$ such that
\begin{equation} \label{TF_conv_HJepsACP}
J[u^\varepsilon](x_\gamma, t_\gamma) + K_{(0, t_\gamma)}[u^\varepsilon](x_\gamma, t_\gamma) + \overline{H} \left( \frac{z_\gamma}{\varepsilon^\beta} \right) \leq C ( |z_\gamma|\varepsilon^{\theta-\beta} + \varepsilon^\theta + \varepsilon^{1 - \theta - \beta}) + E_{1, \gamma},
\end{equation}
for some $E_{1, \gamma} > 0$ satisfying $E_{1, \gamma} \to 0$ as $\gamma \to 0$.
Notice that the function $(x, t) \mapsto \Psi(x, y_\gamma, z_\gamma, \xi_\gamma, t, s_\gamma)$ has a maximum at $(x_\gamma, t_\gamma)$.
Since $u^\varepsilon$ is a viscosity subsolution of (\ref{TF_HJ_eqs}), by Proposition \ref{vissol_equivalent}, we have
\begin{equation} \label{TF_conv_HJ_eps}
\begin{split}
J[u^\varepsilon](x_\gamma, t_\gamma) &+ K_{(0, t_\gamma)}[u^\varepsilon](x_\gamma, t_\gamma) \\
&+ H \left( \frac{x_\gamma}{\varepsilon}, \frac{x_\gamma - y_\gamma}{\varepsilon^\beta} + \frac{x_\gamma - \varepsilon \xi_\gamma + x_\gamma - y_\gamma - z_\gamma}{\gamma} + x_\gamma - \hat{x} \right) \leq 0.
\end{split}
\end{equation}
On the other hand, notice that the function $\xi \mapsto - \frac{1}{\varepsilon} \Psi(x_\gamma, y_\gamma, z_\gamma, \xi, t_\gamma, s_\gamma)$ has a minimum at $\xi_\gamma$.
Since $v^\lambda$ is a viscosity supersolution of $(\ref{TF_ACP_2})$, we obtain 
\begin{equation} \label{TF_conv_ACP}
\lambda v^\lambda \left( \xi_\gamma, \frac{z_\gamma}{\varepsilon^\beta} \right) + H \left( \xi_\gamma, \frac{z_\gamma}{\varepsilon^\beta} + \frac{x_\gamma - \varepsilon \xi_\gamma}{\gamma} \right) \geq 0.
\end{equation}
By $\Psi(x_\gamma, y_\gamma, x_\gamma - y_\gamma, \xi_\gamma, t_\gamma, s_\gamma) \leq \Psi(x_\gamma, y_\gamma, z_\gamma, \xi_\gamma, t_\gamma, s_\gamma)$ and Lemma \ref{TF_ACP_Property}, we see that
\begin{align*}
\frac{|x_\gamma - y_\gamma - z_\gamma|^2}{2\gamma} 
&\leq \varepsilon \left\{ v^\lambda \left( \xi_\gamma, \frac{x_\gamma - y_\gamma}{\varepsilon^\beta} \right) - v^\lambda \left( \xi_\gamma, \frac{z_\gamma}{\varepsilon^\beta} \right) \right\} \\
& \leq \varepsilon C \frac{1}{\lambda} \frac{|x_\gamma - y_\gamma - z_\gamma|}{\varepsilon^\beta} = C \varepsilon^{1 - \theta - \beta} |x_\gamma - y_\gamma - z_\gamma|,
\end{align*}
which implies
\begin{equation} \label{TF_Conv_xyz}
\frac{|x_\gamma - y_\gamma - z_\gamma|}{\gamma} \leq C \varepsilon^{1 - \theta - \beta}.
\end{equation}
Combining (\ref{TF_conv_HJ_eps}) with (\ref{TF_conv_ACP}) and \eqref{TF_Conv_xyz}, by Lemma \ref{TF_ACP_Property}, we obtain
\begin{align*}
& J[u^\varepsilon](x_\gamma, t_\gamma) + K_{(0, t_\gamma)}[u^\varepsilon](x_\gamma, t_\gamma) + \overline{H} \left( \frac{z_\gamma}{\varepsilon^\beta} \right) \\
 \leq& \, 
 \lambda v^\lambda \left(\xi_\gamma, \frac{z_\gamma}{\varepsilon^\beta}\right) + \overline{H} \left( \frac{z_\gamma}{\varepsilon^\beta} \right) 
+ H \left( \xi_\gamma, \frac{z_\gamma}{\varepsilon^\beta} + \frac{x_\gamma - \varepsilon \xi_\gamma}{\gamma} \right)\\
& - H \left( \frac{x_\gamma}{\varepsilon}, \frac{x_\gamma - y_\gamma}{\varepsilon^\beta} + \frac{x_\gamma - \varepsilon \xi_\gamma + x_\gamma - y_\gamma - z_\gamma}{\gamma} + x_\gamma - \hat{x} \right) \\
\leq& \, C \left( 1 + \left| \frac{z_\gamma}{\varepsilon^\beta} \right| \right) \lambda + C \left( \left|\frac{x_\gamma}{\varepsilon} - \xi_\gamma\right| + \frac{|x_\gamma - y_\gamma - z_\gamma|}{\varepsilon^\beta} + \frac{|x_\gamma - y_\gamma - z_\gamma|}{\gamma} + |x_\gamma - \hat{x}| \right) \\
\leq& \, C( |z_\gamma|\varepsilon^{\theta-\beta} + \varepsilon^\theta + \varepsilon^{1 - \theta - \beta}) + C \left( \left|\frac{x_\gamma}{\varepsilon} - \xi_\gamma\right| + \frac{|x_\gamma - y_\gamma - z_\gamma|}{\varepsilon^\beta} + |x_\gamma - \hat{x}| \right ) \\
\revcoloneqq& \, C( |z_\gamma|\varepsilon^{\theta-\beta} + \varepsilon^\theta + \varepsilon^{1 - \theta - \beta}) + E_{1, \gamma}.
\end{align*}

\medskip
STEP 2-2. 
Next, we claim that there exists $C > 0$ such that
\begin{equation} \label{TF_conv_HJ}
J[\overline{u}](y_\gamma, s_\gamma) + K_{(0, s_\gamma)}[\overline{u}](y_\gamma, s_\gamma) + \overline{H} \left( \frac{z_\gamma}{\varepsilon^\beta} \right) \geq - C ( \varepsilon^{1 - \theta - \beta} + \varepsilon^\theta )- E_{2, \gamma} 
\end{equation}
for some $E_{2, \gamma} > 0$ satisfying $E_{2, \gamma} \to 0$ as $\gamma \to 0$.
Notice that $(y, s) \mapsto - \Psi(x_\gamma, y, z_\gamma, \xi_\gamma, t_\gamma, s)$ has a minimum at $(y_\gamma, s_\gamma)$.
By Proposition \ref{vissol_equivalent}, we have
\begin{equation*}
J[\overline{u}](y_\gamma, s_\gamma) + K_{(0, s_\gamma)}[\overline{u}](y_\gamma, s_\gamma) + \overline{H} \left( \frac{x_\gamma - y_\gamma}{\varepsilon^\beta} - \delta y_\gamma + \frac{x_\gamma - y_\gamma - z_\gamma}{\gamma} - (y_\gamma - \hat{y}) \right) \geq 0.
\end{equation*}
Thus, by (\ref{TF_Conv_y}), (\ref{TF_Conv_xyz}) and Lemma \ref{TF_ACP_Property}, we obtain
\begin{align*}
& J[\overline{u}](y_\gamma, s_\gamma) + K_{(0, s_\gamma)}[\overline{u}](y_\gamma, s_\gamma) + \overline{H} \left(\frac{z_\gamma}{\varepsilon^\beta} \right) \\
& \geq \overline{H} \left( \frac{z_\gamma}{\varepsilon^\beta} \right) - \overline{H} \left( \frac{x_\gamma - y_\gamma}{\varepsilon^\beta} - \delta y_\gamma + \frac{x_\gamma - y_\gamma - z_\gamma}{\gamma} - (y_\gamma - \hat{y}) \right) \\
& \geq - C \left\{ \frac{|x_\gamma - y_\gamma - z_\gamma|}{\varepsilon^\beta} + \frac{|x_\gamma - y_\gamma - z_\gamma|}{\gamma} + (1 + \delta)|y_\gamma - \hat{y}| + \delta |\hat{y}| \right\} \\
& \geq - C ( \varepsilon^{1 - \theta - \beta} + \delta^{\frac{1}{2}} ) - \left\{ \frac{|x_\gamma - y_\gamma - z_\gamma|}{\varepsilon^\beta} + (1 + \delta)|y_\gamma - \hat{y}| \right\} \\
& \revcoloneqq - C (\varepsilon^{1 - \theta - \beta} + \delta^{\frac{1}{2}} ) - E_{2, \gamma}.
\end{align*}
Take $0 < \delta < \frac{1}{2}$ satisfying $\delta^{\frac{1}{2}} < \lambda = \varepsilon^\theta$ to get (\ref{TF_conv_HJ}).
Combining (\ref{TF_conv_HJepsACP}) with (\ref{TF_conv_HJ}), we get
\begin{equation*}
\begin{split}
J[u^\varepsilon](x_\gamma, t_\gamma) - J[\overline{u}](y_\gamma, s_\gamma)
+ K_{(0, t_\gamma)}[u^\varepsilon] &(x_\gamma, t_\gamma) - K_{(0, s_\gamma)}[\overline{u}](y_\gamma, s_\gamma) \\
&\leq C (|z_\gamma|\varepsilon^{\theta-\beta} + \varepsilon^\theta + \varepsilon^{1-\theta-\beta}) + E_{1, \gamma} + E_{2, \gamma}.
\end{split}
\end{equation*}
We set $E_\gamma \coloneqq E_{1, \gamma} + E_{2, \gamma}$.
This completes STEP 2.

\medskip
STEP 3. 
Finally, we claim that taking the limit infimum $\gamma \to 0$ in (\ref{TF_Conv_STEP2}) yields that
\begin{equation*}
J[u^\varepsilon](\hat{x}, \hat{t}) - J[\overline{u}](\hat{y}, \hat{t})
+ K_{(0, \hat{t})}[u^\varepsilon](\hat{x}, \hat{t}) - K_{(0, \hat{t})}[\overline{u}](\hat{y}, \hat{t})
\leq C (\varepsilon^{\theta - \frac{\beta(1-\nu)}{2-\nu}} + \varepsilon^\theta + \varepsilon^{1-\theta-\beta}).
\end{equation*}
We set
\begin{align*}
I_1 &\coloneqq J[u^\varepsilon](x_\gamma, t_\gamma) - J[\overline{u}](y_\gamma, s_\gamma) \\
&= \frac{u^\varepsilon(x_\gamma, t_\gamma) - u^\varepsilon(x_\gamma, 0)}{t_\gamma^\alpha \Gamma(1 - \alpha)} - \frac{\overline{u}(y_\gamma, s_\gamma) - \overline{u}(y_\gamma, 0)}{s_\gamma^\alpha \Gamma(1 - \alpha)}, \\
I_2 &\coloneqq \frac{\Gamma(1 - \alpha)}{\alpha} \{ K_{(0, r)}[u^\varepsilon] (x_\gamma, t_\gamma) - K_{(0, r)}[\overline{u}](y_\gamma, s_\gamma) \}\\
&= \int_0^r \frac{u^\varepsilon(x_\gamma, t_\gamma) - u^\varepsilon(x_\gamma, t_\gamma - \tau)}{\tau^{\alpha + 1}} d\tau - \int_0^r \frac{\overline{u}(y_\gamma, s_\gamma) - \overline{u}(y_\gamma, s_\gamma - \tau)}{\tau^{\alpha + 1}} d\tau, \\
I_3 &\coloneqq \frac{\Gamma(1 - \alpha)}{\alpha} \{ K_{(r, t_\gamma)}[u^\varepsilon] (x_\gamma, t_\gamma) - K_{(r, s_\gamma)}[\overline{u}](y_\gamma, s_\gamma) \} \\
&= \int_r^{t_\gamma} \frac{u^\varepsilon(x_\gamma, t_\gamma) - u^\varepsilon(x_\gamma, t_\gamma - \tau)}{\tau^{\alpha + 1}} d\tau - \int_r^{s_\gamma} \frac{\overline{u}(y_\gamma, s_\gamma) - \overline{u}(y_\gamma, s_\gamma - \tau)}{\tau^{\alpha + 1}} d\tau
\end{align*}
for $0 < r < \min\{t_\gamma, s_\gamma\}$.
Notice that $I_1 + \frac{\alpha}{\Gamma(1 - \alpha)}(I_2 + I_3) = J[u^\varepsilon](x_\gamma, t_\gamma) - J[\overline{u}](y_\gamma, s_\gamma)
+ K_{(0, t_\gamma)}[u^\varepsilon] (x_\gamma, t_\gamma) - K_{(0, s_\gamma)}[\overline{u}](y_\gamma, s_\gamma)$.

First, by the continuity of $u^\varepsilon$ and $\overline{u}$, we have
\begin{equation} \label{TF_Conv_I1}
\lim_{\gamma \to 0} I_1 = J[u^\varepsilon](\hat{x}, \hat{t}) - J[\overline{u}](\hat{y}, \hat{t}).
\end{equation}
Next, by $\Psi(x_\gamma, y_\gamma, z_\gamma, \xi_\gamma, t_\gamma, s_\gamma) \geq \Psi(x_\gamma, y_\gamma, z_\gamma, \xi_\gamma, t_\gamma - \tau, s_\gamma - \tau)$ for all $\tau \in [0, r]$, i.e.,
\begin{equation*}
\begin{split}
u^\varepsilon(x_\gamma, t_\gamma) - \overline{u}(y_\gamma, s_\gamma) - &\frac{R t_\gamma^\alpha}{\Gamma(1 + \alpha)} - \frac{|t_\gamma - \hat{t}|^2}{2} \\
&\geq
u^\varepsilon(x_\gamma, t_\gamma - \tau) - \overline{u}(y_\gamma, s_\gamma - \tau) - \frac{R (t_\gamma - \tau)^\alpha}{\Gamma(1 + \alpha)} - \frac{|t_\gamma - \tau - \hat{t}|^2}{2},
\end{split}
\end{equation*}
we have
\begin{align}
I_2  
&\geq \int_0^r \left\{ R\frac{t_\gamma^\alpha - (t_\gamma - \tau)^\alpha}{\Gamma(1 + \alpha)} + \frac{|t_\gamma - \hat{t}|^2 - |t_\gamma - \tau - \hat{t}|^2}{2} \right\} \frac{d\tau}{\tau^{\alpha + 1}} 
\nonumber\\
&\geq \int_0^r \frac{|t_\gamma - \hat{t}|^2 - |t_\gamma - \tau - \hat{t}|^2}{2} \frac{d\tau}{\tau^{\alpha + 1}} 
= \int_0^r \frac{2(t_\gamma - \hat{t})\tau - \tau^2}{2} \frac{d\tau}{\tau^{\alpha + 1}} 
\nonumber
\\
&= \frac{(t_\gamma - \hat{t}) r^{1 - \alpha}}{1 - \alpha} - \frac{r^{2 - \alpha}}{2 (2 - \alpha)} \to - \frac{r^{2 - \alpha}}{2 (2 - \alpha)} \revcoloneqq - C_r \quad \mbox{as} \,\, \gamma \to 0,
 \label{TF_Conv_I2}
 \end{align}

Finally, we see that
\begin{align*}
I_3 
&= \int_r^T \left( \frac{u^\varepsilon(x_\gamma, t_\gamma) - u^\varepsilon(x_\gamma, t_\gamma - \tau)}{\tau^{\alpha + 1}} \chi_{(r, t_\gamma)}(\tau) - \frac{\overline{u}(y_\gamma, s_\gamma) - \overline{u}(y_\gamma, s_\gamma - \tau)}{\tau^{\alpha + 1}} \chi_{(r, s_\gamma)}(\tau) \right) d\tau,
\end{align*}
where $\chi_I$ is the characteristic function. 
Notice that by Proposition \ref{TF_uni_time_regularity}
\begin{equation*}
\begin{split}
\frac{u^\varepsilon(x_\gamma, t_\gamma) - u^\varepsilon(x_\gamma, t_\gamma - \tau)}{\tau^{\alpha + 1}} \chi_{(r, t_\gamma)}(\tau) - \frac{\overline{u}(y_\gamma, s_\gamma) - \overline{u}(y_\gamma, s_\gamma - \tau)}{\tau^{\alpha + 1}} \chi_{(r, s_\gamma)}(\tau) \\
\geq
- \frac{2C}{\tau} \chi_{(r, T)}(\tau).
\end{split}
\end{equation*}
Since the right-hand side is integrable on $[0, T]$, by Fatou's lemma, we obtain
\begin{equation} \label{TF_Conv_I3}
\begin{split}
\liminf_{\gamma \to 0} I_3 &\geq \int_r^{\hat{t}} \frac{u^\varepsilon(\hat{x}, \hat{t}) - u^\varepsilon(\hat{x}, \hat{t} - \tau)}{\tau^{\alpha + 1}} d\tau - \int_r^{\hat{t}} \frac{\overline{u}(\hat{y}, \hat{t}) - \overline{u}(\hat{y}, \hat{t} - \tau)}{\tau^{\alpha + 1}} d\tau \\
&= \frac{\Gamma(1 - \alpha)}{\alpha} ( K_{(r, \hat{t})}[u^\varepsilon](\hat{x}, \hat{t}) - K_{(r, \hat{t})}[\overline{u}](\hat{y}, \hat{t}) ).
\end{split}
\end{equation}

Combining (\ref{TF_Conv_I1}), (\ref{TF_Conv_I2}) with (\ref{TF_Conv_I3}), and sending $\gamma \to 0$ in (\ref{TF_Conv_STEP2}), by \eqref{TF_Conv_diff_xy_pre}
 we obtain 
\begin{align*}
&J[u^\varepsilon](\hat{x}, \hat{t}) - J[\overline{u}](\hat{y}, \hat{t})
- \frac{\alpha C_r}{\Gamma(1 - \alpha)}
+ K_{(r, \hat{t})}[u^\varepsilon](\hat{x}, \hat{t}) - K_{(r, \hat{t})}[\overline{u}](\hat{y}, \hat{t}) \\
&\leq C (|\hat{x} - \hat{y}|\varepsilon^{\theta - \beta} + \varepsilon^\theta + \varepsilon^{1-\theta-\beta})
\le C(\ep^{\theta-\frac{\beta(1-\nu)}{2-\nu}}+\ep^\theta+\ep^{1-\theta-\beta}).
\end{align*}
Note that as in the proof of \eqref{eq:bdd} in Proposition \ref{vissol_equivalent}, we can prove 
\[
K_{(r, \hat{t})}[u^\varepsilon](\hat{x}, \hat{t}) - K_{(r, \hat{t})}[\overline{u}](\hat{y}, \hat{t})
\ge -C
\]
for some $C>0$ independent of $r$. Since $C_r \to 0$ as $r \to 0$ by \eqref{TF_Conv_I2}, 
we obtain the desired result.
\end{proof}

\begin{proof} [Proof of Theorem {\rm\ref{TF_conv_Main_Theorem}}.]
Let $v^\lambda$ be the viscosity solution of (\ref{TF_ACP_2}).
We consider the auxiliary function $\Phi : \mathbb{R}^{2N}\times [0, T] \mapsto \mathbb{R}$ defined by
\begin{equation*}
\Phi(x, y, t) \coloneqq u^\varepsilon(x, t) - \overline{u}(y, t) - \varepsilon v^{\lambda} \left( \frac{x}{\varepsilon}, \frac{x - y}{\varepsilon^\beta} \right) - \frac{|x-y|^2}{2\varepsilon^\beta} - \frac{\delta |y|^2}{2} - \frac{R t^\alpha}{\Gamma(1 + \alpha)},
\end{equation*}
where $\lambda = \varepsilon^\theta, \theta, \beta \in (0, 1), 0 < \delta < \lambda^2$ and $R > 0$.
Noting that $\Phi$ is a bounded from above and $\Phi(x, y, t) \to -\infty$ as $|x|, |y| \to \infty$ for any $t \in [0, T]$, 
we see that 
$\Phi$ attains a maximum at $(\hat{x}, \hat{y}, \hat{t}) \in \mathbb{R}^{2N}\times [0, T]$.

\medskip
First, we set $R:= R^\prime (\varepsilon^{\theta - \frac{\beta(1-\nu)}{2-\nu}}+ \varepsilon^{1-\theta-\beta})$.
We claim that if $R^\prime$ is sufficiently large, then $\hat{t} = 0$.
Suppose that $\hat{t} > 0$.
By Lemma \ref{TF_Conv_Lem}, we have
\begin{equation} \label{TF_Conv_Lem_ineq}
J[u^\varepsilon](\hat{x}, \hat{t}) - J[\overline{u}](\hat{y}, \hat{t})
+ K_{(0, \hat{t})}[u^\varepsilon](\hat{x}, \hat{t}) - K_{(0, \hat{t})}[\overline{u}](\hat{y}, \hat{t})
\leq C (\varepsilon^{\theta - \frac{\beta(1-\nu)}{2-\nu}}+ \varepsilon^{1-\theta-\beta}).
\end{equation}
By $\Phi(\hat{x}, \hat{y}, 0) \leq \Phi(\hat{x}, \hat{y}, \hat{t})$ and $\Phi(\hat{x}, \hat{y}, \hat{t} - \tau) \leq \Phi(\hat{x}, \hat{y}, \hat{t})$ for all $\tau \in [0, \hat{t}]$, 
we obtain
\begin{align*}
&J[u^\varepsilon](\hat{x}, \hat{t}) - J[\overline{u}](\hat{y}, \hat{t})\geq \frac{R}{\Gamma(1 - \alpha) \Gamma(1 + \alpha)}, \\
& K_{(0, \hat{t})}[u^\varepsilon](\hat{x}, \hat{t}) - K_{(0, \hat{t})}[\overline{u}](\hat{y}, \hat{t})
\geq \frac{\alpha R}{\Gamma(1 - \alpha) \Gamma(1 + \alpha)} \int_0^{\hat{t}} \frac{\hat{t}^\alpha - (\hat{t} - \tau)^\alpha}{\tau^{\alpha + 1}} d\tau \geq 0.
\end{align*}
Combining two inequalities above with \eqref{TF_Conv_Lem_ineq}, we have
\begin{equation*}
R \leq C \Gamma(1 - \alpha) \Gamma(1 + \alpha) (\varepsilon^{\theta - \frac{\beta(1-\nu)}{2-\nu}}+ \varepsilon^{1-\theta-\beta}).
\end{equation*}
Thus, we set $R^\prime > C \Gamma(1 - \alpha) \Gamma(1 + \alpha)$, which implies a contradiction.

By a similar argument to the proof of Lemma \ref{TF_Conv_Lem}, we have 
$|\hat{x} - \hat{y}| \leq C \varepsilon^\beta$ 
if $\delta < \frac{1}{2}$ and $0 < \theta < 1 - \beta$. 

Next, by $\Phi(x, x, t) \leq \Phi(\hat{x}, \hat{y}, \hat{t}) = \Phi(\hat{x}, \hat{y}, 0)$ for all $(x, t) \in \mathbb{R}^N \times (0, T)$, 
we obtain, by Lemma \ref{TF_ACP_Property}, 
\begin{align*}
&u^\varepsilon (x, t) - \overline{u}(x, t) \\
&\leq u^\varepsilon (\hat{x}, 0) - \overline{u}(\hat{y}, 0) + \varepsilon \left\{ v^{\lambda} \left(\frac{x}{\varepsilon}, 0 \right) - v^{\lambda} \left(\frac{\hat{x}}{\varepsilon}, \frac{\hat{x} - \hat{y}}{\varepsilon^\beta} \right) \right\} + \frac{\delta |x|^2}{2} + \frac{R t^\alpha}{\Gamma(1 + \alpha)} \\
&\leq \Lip[u_0] |\hat{x} - \hat{y}| + \frac{C\varepsilon}{\lambda} \left( \frac{|\hat{x} - \hat{y}|}{\varepsilon^\beta} + 1 \right) + \frac{\delta |x|^2}{2} + (\varepsilon^{\theta- \frac{\beta(1-\nu)}{2-\nu}}  + \varepsilon^{1 - \theta - \beta}) \frac{R^\prime T^\alpha}{\Gamma(1 + \alpha)} \\
&\leq C (\varepsilon^\beta + \varepsilon^{1-\theta} + \varepsilon^{\theta - \frac{\beta(1-\nu)}{2-\nu}}+ \varepsilon^{1 - \theta - \beta}) + \frac{\delta |x|^2}{2}
\end{align*}
for all $(x, t) \in \mathbb{R}^N \times (0, T)$.
Therefore, by the optimal choice of parameters $\theta = \beta = \frac{1}{3}$, and sending $\delta \to 0$, we obtain
\begin{equation*}
u^\varepsilon(x, t) - \overline{u}(x, t) \leq C \varepsilon^{\frac{1}{3(2-\nu)}}.
\end{equation*}
By a symmetric argument, we get the desired result.
\end{proof}


\subsection*{Acknowledgments}
The authors would like to thank Olivier Ley for helpful comments and suggestions. 


\end{document}